\documentclass[10pt,a4paper]{article}
\usepackage{amsmath,amssymb,amsthm,graphics,graphicx,epic,eepic,multicol,ascmac}

\numberwithin{equation}{section}
\theoremstyle{theorem}
\newtheorem{thm}{Theorem}[section]
\newtheorem{prop}[thm]{Proposition}
\newtheorem{lem}[thm]{Lemma}
\newtheorem{rem}[thm]{Remark}

\newtheorem{ex}[thm]{Example}

\theoremstyle{definition}
\newtheorem{defn}[thm]{Definition}

\def\al{\alpha}

\def\ep{\epsilon}
\def\wht(#1){\widehat{\ #1\ }}

\newcommand{\cA}{{\mathcal A}}

\newcommand{\cF}{{\mathcal F}}

\newcommand{\cY}{{\mathcal Y}}

\newcommand{\frg}{\mathfrak g}
\newcommand{\frh}{\mathfrak h}

\newcommand{\frl}{\mathfrak l}

\newcommand{\frn}{\mathfrak n}
\newcommand{\fro}{\mathfrak o}
\newcommand{\frp}{\mathfrak p}

\newcommand{\frs}{\mathfrak s}

\newcommand{\bbC}{\mathbb C}

\newcommand{\ch}{\mathrm{ch}}

\newcommand{\lbr}{\begin{bmatrix}}
\newcommand{\rbr}{\end{bmatrix}}

\newcommand{\cd}{commutative diagram }

\def\ge{\frg}

\def\cP{{\mathcal P}}
\def\al{\alpha}

\def\beneme{\begin{enumerate}}
\def\beq{\begin{equation}}
\def\beqn{\begin{eqnarray}}
\def\beqnn{\begin{eqnarray*}}
\def\bfi{{\mathbf i}}
\def\bfii0{{\bf i_0}}

\def\bbra#1,#2,#3{\left\{\begin{array}{c}\hspace{-5pt}
#1;#2\\ \hspace{-5pt}#3\end{array}\hspace{-5pt}\right\}}
\def\cd{\cdots}
\def\ci(#1,#2){c_{#1}^{(#2)}}
\def\Ci(#1,#2){C_{#1}^{(#2)}}
\def\mpp(#1,#2,#3){#1^{(#2)}_{#3}}
\def\bCi(#1,#2){\ovl C_{#1}^{(#2)}}
\def\ch(#1,#2){c_{#2,#1}^{-h_{#1}}}
\def\cc(#1,#2){c_{#2,#1}}

\def\del{\delta}
\def\Del{\Delta}

\def\di(#1,#2){D_{#1}^{(#2)}}
\def\dbi(#1,#2){\ovl D_{#1}^{(#2)}}

\def\eneme{\end{enumerate}}
\def\ep{\epsilon}
\def\eeq{\end{equation}}
\def\eeqn{\end{eqnarray}}
\def\eeqnn{\end{eqnarray*}}

\def\gau#1,#2{\left[\begin{array}{c}\hspace{-5pt}#1\\
\hspace{-5pt}#2\end{array}\hspace{-5pt}\right]}

\def\ji(#1,#2){j_{#1}^{(#2)}}

\def\lan{\langle}

\def\lm{\lambda}
\def\Lm{\Lambda}
\def\mapright#1{\smash{\mathop{\longrightarrow}\limits^{#1}}}

\def\nd{\noindent}

\def\ovl{\overline}

\def\qq{\qquad}
\def\q{\quad}
\def\qed{\hfill\framebox[2mm]{}}

\def\ran{\rangle}

\def\TY(#1,#2,#3){#1^{(#2)}_{#3}}

\def\xxi(#1,#2,#3){\displaystyle {}^{#1}\Xi^{(#2)}_{#3}}
\def\xsi(#1,#2,#3){\displaystyle {}^{#1}\Sigma^{(#2)}_{#3}}
\def\xE(#1,#2,#3){\displaystyle {}^{#1}E_{#2}[#3]}
\def\xF(#1,#2){\displaystyle {}^{#1}F_{#2}}
\def\xx(#1,#2){\displaystyle {}^{#1}\Xi_{#2}}
\def\W1{W(\varpi_1)}

\def\m@th{\mathsurround=0pt}
\def\fsquare(#1,#2){
\hbox{\vrule$\hskip-0.4pt\vcenter to #1{\normalbaselines\m@th
\hrule\vfil\hbox to #1{\hfill$\scriptstyle #2$\hfill}\vfil\hrule}$\hskip-0.4pt
\vrule}}

\newcommand{\ba}{\begin{array}}
\newcommand{\ea}{\end{array}}

\newcommand{\eq}{\begin{eqnarray}}
\newcommand{\eneq}{\end{eqnarray}}

\title{\textbf{\large{Cluster Variables 
on Double Bruhat Cells $G^{u,e}$ of Classical Groups and 
Monomial Realizations of 
Demazure Crystals}}}
\author{\normalsize{YUKI KANAKUBO\thanks{Division of Mathematics, 
Sophia University, Kioicho 7-1, Chiyoda-ku, Tokyo 102-8554,
Japan: {j\_chi\_sen\_you\_ky@eagle.sophia.ac.jp}}
\ and\ 
TOSHIKI NAKASHIMA\thanks{Division of Mathematics, 
Sophia University, Kioicho 7-1, Chiyoda-ku, Tokyo 102-8554,
Japan: { toshiki@sophia.ac.jp}:
supported in part by JSPS Grants 
in Aid for Scientific Research $\sharp 15K04794$.}
}}
\date{}

\begin{document}

\maketitle
\vspace{-10pt}

\begin{abstract}
Let $G$ be a simply connected simple algebraic group over $\mathbb{C}$, 
$B$ and $B_-$ its two opposite Borel subgroups, and $W$ the associated 
Weyl group. 
It is shown that the coordinate ring 
${\mathbb C}[G^{u,v}]$ ($u$, $v\in W$)
of the double Bruhat cell $G^{u,v}=BuB\cap B_-vB_-$ is 
isomorphic to the cluster algebra ${\cA}(\textbf{i})_{{\mathbb C}}$ 
and the initial cluster variables of ${\mathbb C}[G^{u,v}]$ are 
the generalized minors $\Delta(k;\textbf{i})$
(\cite{A-F-Z:2005,GY:2016}). 
In the case that a classical group $G$ is of 
type ${\rm B}_r$, ${\rm C}_r$ or ${\rm D}_r$, 
we shall describe the non-trivial last $r$ initial cluster variables 
$\{\Delta(k;\textbf{i})\}_{(m-2)r<k\leq (m-
1)r}$ ($m$ is given below) 
of the cluster algebra $\bbC[L^{u,e}]$ (\cite{GLS:2011,GY:2016})
in terms of monomial realization of Demazure crystals,
where $L^{u,e}$ is the reduced double Bruhat
cell of type $(u,e)$.
The relation between $\Delta(k;\textbf{i})$ on $G^{u,e}$ and on $L^{u,e}$ 
is described in Proposition \ref{gprop} below.
We also present the corresponding results for type ${\rm A}_r$ 
though the results for all initial cluster variables have been obtained
in \cite{KaN:2015}.
\end{abstract}



\section{Introduction}
In \cite{FZ2:2002,FZ3:2003} Fomin and Zelevinsky initiated the
theory of cluster algebras, which is a commutative algebra 
generated by so-called ``cluster variables''. 

Let $G$ be a simply connected complex simple algebraic group of rank
$r$, $B,B_-\subset G$ the opposite Borel subgroups, 
$H:=B\cap B_-$ the maximal torus, 
$N\subset B$, $N_-\subset B_-$ the unipotent radical
and $W=\lan s_i|1\leq  i\leq  r\ran$ 
the associated Weyl group generated by the simple reflections 
$\{s_i\}_{1\leq i\leq r}$. 
For $u,v\in W$, define the (reduced) double Bruhat cell 
$G^{u,v}:=(Bu B)\cap(B_- v B_-)$ (resp. $L^{u,v}:=(Nu N)\cap(B_- v B_-)$).
In \cite{A-F-Z:2005}, it is revealed that 
 the coordinate ring $\bbC[G^{u,v}]$ ($u,v\in W$) of double Bruhat
cell $G^{u,v}$ is isomorphic to the upper cluster algebra
$\ovl\cA(\bfi)_{\bbC}$. 
Recently, in \cite{GY:2016}, Goodear and Yakimov have shown that the algebra
$\bbC[G^{u,v}]$ ($u,v\in W$) is isomorphic to the cluster algebra
$\cA(\bfi)_\bbC$ (see Subsection \ref{UCACA} and \ref{CADBC}):
\begin{eqnarray*}
\phi:\,\, \cA(\bfi)_\bbC&\mapright{\sim}&\bbC[G^{u,v}],\\
x_k&\longrightarrow& \Del(k;\bfi),
\end{eqnarray*}
where $\bfi$ is a reduced word for the shuffle $(u,v)$ and 
$\Del(k;\bfi)$ is certain generalized minor on $G^{u,v}$. It means that
the $k$-th initial cluster variable of the cluster algebra 
$\bbC[G^{u,v}]$ is given
as certain generalized minors $\Del(k;\bfi)$ on $G^{u,v}$.

In \cite{KaN:2015}, for type ${\rm A}_r$ and a pair $(u,e)$ with the 
specific reduced word $\bfi$  
we gave the explicit forms of these initial 
cluster variables $\{\Del(k;\bfi)\}$ and described them by using 
monomial realization of certain Demazure crystals.

In this article, we shall get the partial results 
for the classical groups of type ${\rm B}_r, {\rm C}_r$ and ${\rm D}_r$.
Indeed, we find that 
the non-trivial last $r$ initial cluster variables are expressed by 
using monomial realizations of Demazure crystals.
Let us explain what we obtain here in more details.
As the reduced longest word let us take 
 $\bfi_0=(12\cd r)^{r'}$ ($r'=r$ for ${\rm B}_r$, ${\rm C}_r$, and  $r'=r-1$
for ${\rm D}_r$), and 
set $\bfi$ a left factor of $\bfi_0$ whose length is 
in $[(m-1)r+1,mr]$ ($m\leq r'$), $u$ the
corresponding Weyl group element and $(m-2)r<k\leq (m-1)r$.
Let $x^L_\bfi(t)$ be as in \eqref{xldef} and $\Del^L(k;\bfi)$ as in 
Definition \ref{gendef}. 
For example, in Example \ref{monorealex2} 
for type ${\rm C}_2$ we have 
\begin{equation}
\Del^L(2;\bfi)(Y)=
\frac{Y^2_{1,1}}{Y_{1,2}}+2\frac{Y_{1,1}}{Y_{2,1}} 
+\frac{Y_{1,2}}{Y^2_{2,1}}+\frac{1}{Y_{2,2}}
\label{intro1}
\end{equation}
where  $\bfi=(1,2,1,2)$ and $Y=(Y_{1,1},Y_{1,2},Y_{2,1},Y_{2,2})$.
On the other hand, we have the crystal graph $B(\Lm_2)$ of type 
${\rm C}_2$ as in Example \ref{monorealex1}:
\begin{equation}
Y_{0,2}\overset{2}{\longrightarrow}\frac{Y^2_{1,1}}{Y_{1,2}}
\overset{1}{\longrightarrow}\frac{Y_{1,1}}{Y_{2,1}} 
\overset{1}{\longrightarrow}\frac{Y_{1,2}}{Y^2_{2,1}}
\overset{2}{\longrightarrow}\frac{1}{Y_{2,2}}.
\label{intro2}
\end{equation}
Comparing \eqref{intro1} and \eqref{intro2}, 
it is not difficult to find their relations, 
that is, $\Del^L(2;\bfi)$ is expressed as a summation of
 all vertices but $Y_{0,2}$
of the graph, which is what we want to clarify in this article.
Here we also observe the feature different from the one for type ${\rm A}_r$.
It is that coefficients greater than 1 appear in $\Del^L(k;\bfi)$,
which is possible to occur for types ${\rm B}_r$, ${\rm C}_r$ and ${\rm D}_r$.
In the case $k\leq (m-2)r$, indeed, the generalized minors $\Del(k;\bfi)$
seem to be expressed by monomial realizations of certain subset of crystals, 
nevertheless, not necessarily  Demazure crystals. This is the reason that 
here we only treat the non-trivial last $r$ generalized minors. 
So we expect to find a new characterization 
to abstract a subset of crystals matching 
the generalized minors, which is a further task for us.

Let us see the organization of this article.
In Sect.2, we review the explicit forms of fundamental representations
of classical groups. In Sect.3, we introduce double Bruhat cells and 
in Sect.4, we shall review the notion of cluster algebras and 
generalized minors. Sect.5 is devoted to recall the theory of crystals 
and their monomial realizations.
In Sect.6, we state Theorem \ref{thm1} 
the main results of the article, which claims as follows:
suppose that $u$ is a Weyl group element corresponding to
a left factor of $\bfi_0$ with $m(r-2)< k\leq m(r-1)< l(u)\leq mr$ and $v=e$. 
In the setting, 
$\Del^L(k;\bfi)$ is expressed by a summation of 
monomials in some Demazure crystals (see Subsection \ref{demazure}).
In Sect.7, 
the explicit forms of $\Del^L(k;\bfi)$ are described
using the results in \cite{KaN:2015,KaN2:2016,Ka1:2016}.
In the last section, the proof of the main theorem is presented.

\section{Fundamental representations}\label{SectFund}

First, we review the fundamental representations of the complex simple Lie algebras $\ge$ of type ${\rm A}_r$,  ${\rm B}_r$,  ${\rm C}_r$,  and  ${\rm D}_r$ \cite{KN:1994, N1:2014}. We shall use them in calculations of generalized minors (see Subsection \ref{bilingen}). 
Let $I:=\{1,\cdots,r\}$ be a finite index set, $A=(a_{ij})_{i,j\in I}$ 
be the Cartan matrix of $\ge$, and $(\frh,\{\al_i\}_{i\in I},\{h_i\}_{i\in I})$ 
be the associated
root data 
satisfying $\al_j(h_i)=a_{ij}$ where 
$\al_i\in \frh^*$ is a simple root and 
$h_i\in \frh$ is a simple co-root.
Let $\{\Lm_i\}_{i\in I}$ be the set of the fundamental 
weights satisfying $\Lm_i(h_j)=\del_{i,j}$, $P=\bigoplus_{i\in I}\mathbb{Z}\Lm_i$ the weight lattice and $P^*=\bigoplus_{i\in I}\mathbb{Z}h_i$ the dual weight lattice. 

\subsection{Type ${\rm A}_r$}\label{SectFundA}

Let $\frg=\frs\frl(r+1,\mathbb{C})$ be the simple Lie algebra of type ${\rm A}_r$. The Cartan matrix $A=(a_{i,j})_{i,j\in I}$ of $\frg$ is as follows:
\[a_{i,j}=
\begin{cases}
2 & {\rm if}\ i=j, \\
-1 & {\rm if}\ |i-j|=1, \\
0 & {\rm otherwise.}  
\end{cases}
\]
For $\frg=\lan \frh,e_i,f_i(i\in I)\ran$, 
let us describe the vector representation 
$V(\Lm_1)$. Set ${\mathbf B}^{(r)}:=
\{v_i|\ i=1,2,\cd,r+1\}$ and define 
$V(\Lm_1):=\bigoplus_{v\in{\mathbf B}^{(r)}}\bbC v$. The weights of $v_i$ $(i=1,\cd,r+1)$ are given by ${\rm wt}(v_i)=\Lm_i-\Lm_{i-1}$, where $\Lm_0=\Lm_{r+1}=0$. We define the $\frg$-action on $V(\Lm_1)$ as follows:
\begin{eqnarray}
&& h v_j=\lan h,{\rm wt}(v_j)\ran v_j\ \ (h\in P^*,\ j\in J), \\
&&f_iv_i=v_{i+1},\q
e_iv_{i+1}=v_i \q(1\leq i\leq r),\label{A-f1}
\end{eqnarray}
and the other actions are trivial.

Let $\Lm_i$ be the $i$-th fundamental weight of type ${\rm A}_r$.
As is well-known that the fundamental representation 
$V(\Lm_i)$ $(1\leq i\leq r)$
is embedded in $\wedge^i V(\Lm_1)$
with multiplicity free.
The explicit form of the highest (resp. lowest) weight 
vector $u_{\Lm_i}$ (resp. $v_{\Lm_i}$)
of $V(\Lm_i)$ is realized in 
$\wedge^i V(\Lm_1)$ as follows:
\begin{equation}\label{A-h-l}
u_{\Lm_i}=v_1\wedge v_2\wedge\cdots\wedge v_i,\qq
v_{\Lm_i}=v_{i+1}\wedge v_{i+2}\wedge\cdots\wedge v_{r+1}.
\end{equation}

\subsection{Type ${\rm C}_r$}\label{SectFundC}

Let $\frg=\frs\frp(2r,\mathbb{C})$ be the simple Lie algebra of type ${\rm C}_r$. The Cartan matrix $A=(a_{i,j})_{i,j\in I}$ of $\frg$ is given by
\[a_{i,j}=
\begin{cases}
2 & {\rm if}\ i=j, \\
-1 & {\rm if}\ |i-j|=1\ {\rm and}\ (i,j)\neq (r-1,r), \\
-2 & {\rm if}\ (i,j)=(r-1,r), \\
0 & {\rm otherwise.}  
\end{cases}
\]
Note that $\al_i\ (i\neq r)$ are short roots and $\al_r$ is the long simple root. 

Define the total order on the set $J_{{\rm C}}:=\{i,\ovl i|1\leq i\leq r\}$ by 
\begin{equation}\label{C-order}
 1< 2<\cd< r-1< r
< \ovl r< \ovl{r-1}< \cd< \ovl
 2< \ovl 1.
\end{equation}
For $\frg=\lan \frh,e_i,f_i(i\in I)\ran$, 
let us describe the vector representation 
$V(\Lm_1)$. Set ${\mathbf B}^{(r)}:=
\{v_i,v_{\ovl i}|i=1,2,\cd,r\}$ and define 
$V(\Lm_1):=\bigoplus_{v\in{\mathbf B}^{(r)}}\bbC v$. The weights of $v_i$, $v_{\ovl{i}}$ $(i=1,\cd,r)$ are given by ${\rm wt}(v_i)=\Lm_i-\Lm_{i-1}$ and ${\rm wt}(v_{\ovl{i}})=\Lm_{i-1}-\Lm_{i}$,
where $\Lm_0=0$. We define the $\frg$-action on $V(\Lm_1)$ as follows:
\begin{eqnarray}
&& h v_j=\lan h,{\rm wt}(v_j)\ran v_j\ \ (h\in P^*,\ j\in J_{{\rm C}}), \\
&&f_iv_i=v_{i+1},\ f_iv_{\ovl{i+1}}=v_{\ovl i},\q
e_iv_{i+1}=v_i,\ e_iv_{\ovl i}=v_{\ovl{i+1}}
\q(1\leq i<r),\label{C-f1}\\
&&f_r v_r=v_{\ovl r},\qq 
e_r v_{\ovl r}=v_r,\label{C-f2}
\end{eqnarray}
and the other actions are trivial.

Let $\Lm_i$ be the $i$-th fundamental weight of type ${\rm C}_r$.
As is well-known that the fundamental representation 
$V(\Lm_i)$ $(1\leq i\leq r)$
is embedded in $\wedge^i V(\Lm_1)$
with multiplicity free.
The explicit form of the highest (resp. lowest) weight 
vector $u_{\Lm_i}$ (resp. $v_{\Lm_i}$)
of $V(\Lm_i)$ is realized in 
$\wedge^i V(\Lm_1)$ as follows:
\begin{equation}
\begin{array}{ccc}\displaystyle
u_{\Lm_i}&=&v_1\wedge v_2\wedge\cdots\wedge v_i,\\
v_{\Lm_i}&=&v_{\ovl{1}}\wedge v_{\ovl{2}}\wedge\cdots \wedge v_{\ovl{i}}.
\end{array}
\label{C-h-l}
\end{equation}

\subsection{Type ${\rm B}_r$}\label{SectFundB}

Let $\frg=\frs\fro(2r+1,\mathbb{C})$ be the simple Lie algebra of type ${\rm B}_r$. The Cartan matrix $A=(a_{i,j})_{i,j\in I}$ of $\frg$ is given by
\[a_{i,j}=
\begin{cases}
2 & {\rm if}\ i=j, \\
-1 & {\rm if}\ |i-j|=1\ {\rm and}\ (i,j)\neq (r,r-1), \\
-2 & {\rm if}\ (i,j)=(r,r-1), \\
0 & {\rm otherwise.}  
\end{cases}
\]
Note that $\al_i\ (i\neq r)$ are long roots and $\al_r$ is the short simple root. 

Define the total order on the set $J_{{\rm B}}:=\{i,\ovl i|1\leq i\leq r\}\cup\{0\}$ by 
\begin{equation}\label{B-order}
 1< 2<\cd< r-1< r< 0<
 \ovl r< \ovl{r-1}< \cd< \ovl
 2< \ovl 1.
\end{equation}
For $\frg=\lan \frh,e_i,f_i(i\in I)\ran$, 
let us describe the vector representation 
$V(\Lm_1)$. Set ${\mathbf B}^{(r)}:=
\{v_i,v_{\ovl i}|i=1,2,\cd,r\}\cup\{v_0\}$ and define 
$V(\Lm_1):=\bigoplus_{v\in{\mathbf B}^{(r)}}\bbC v$. The weights of $v_i$, $v_{\ovl{i}}$ $(i=1,\cd,r)$ and $v_0$ are as follows:
\begin{equation}\label{B-wtv}
 {\rm wt}(v_i)=\Lm_i-\Lm_{i-1},\q {\rm wt}(v_{\ovl{i}})=\Lm_{i-1}-\Lm_{i}\q (1\leq i\leq r-1),
\end{equation}
\[  {\rm wt}(v_r)=2\Lm_r-\Lm_{r-1},\q {\rm wt}(v_{\ovl{r}})=\Lm_{r-1}-2\Lm_r,\q {\rm wt}(v_0)=0, \]
where $\Lm_0=0$. We define the $\frg$-action on $V(\Lm_1)$ as follows:
\begin{eqnarray}
&& h v_j=\lan h,{\rm wt}(v_j)\ran v_j\ \ (h\in P^*,\ j\in J_{{\rm B}}), \\
&&f_iv_i=v_{i+1},\ f_iv_{\ovl{i+1}}=v_{\ovl i},\q
e_iv_{i+1}=v_i,\ e_iv_{\ovl i}=v_{\ovl{i+1}}
\q(1\leq i<r),\label{B-f1}\\
&&f_r v_r=v_{0},\qq 
e_r v_{\ovl r}=v_0,\qq f_r v_0=2v_{\ovl r},\qq e_r v_0=2v_r,\label{B-f2}
\end{eqnarray}
and the other actions are trivial.

Let $\Lm_i$ $(1\leq i\leq r-1)$ be the $i$-th fundamental weight of type ${\rm B}_r$.
Similar to the ${\rm C}_r$ case, the fundamental representation 
$V(\Lm_i)$
is embedded in $\wedge^i V(\Lm_1)$
with multiplicity free. In 
$\wedge^i V(\Lm_1)$, the highest (resp. lowest) weight 
vector $u_{\Lm_i}$ (resp. $v_{\Lm_i}$)
of $V(\Lm_i)$ is realized as the same form as in (\ref{C-h-l}).

The fundamental representation $V(\Lm_r)$ is called the {\it spin representation}. It can be realized as follows: Set 
\[ {\mathbf B}_{{\rm sp}}^{(r)}:=
\{(\ep_1,\cdots,\ep_r)|\ \ep_i\in\{+,- \}\q (i=1,2,\cd,r)\}, \]
\[ V_{{\rm sp}}^{(r)}:=\bigoplus_{v\in{\mathbf B}_{{\rm sp}}^{(r)}}\bbC v,\] 
and define the $\ge$-action on $V_{{\rm sp}}^{(r)}$ as follows:

\begin{equation}\label{Bsp-f0}
h_i(\ep_1,\cdots,\ep_r)=
\begin{cases}
\frac{\ep_i\cdot 1-\ep_{i+1}\cdot 1}{2}(\ep_1,\cdots,\ep_r) & {\rm if}\ i<r, \\
\ep_r(\ep_1,\cdots,\ep_r) & {\rm if}\ i=r,
\end{cases}
\end{equation}

\begin{equation}\label{Bsp-f1}
f_i(\ep_1,\cdots,\ep_r)=
\begin{cases}
(\ep_1,\cdots,\overset{i}{-},\overset{i+1}{+},\cdots,\ep_r) & {\rm if}\ \ep_i=+,\ \ep_{i+1}=-,\ i\neq r, \\
(\ep_1,\cdots,\ep_{r-1},\overset{r}{-}) & {\rm if}\ \ep_r=+,\ i=r, \\
0 & {\rm otherwise,}
\end{cases}
\end{equation}

\begin{equation}\label{Bsp-f2}
e_i(\ep_1,\cdots,\ep_r)=
\begin{cases}
(\ep_1,\cdots,\overset{i}{+},\overset{i+1}{-},\cdots,\ep_r) & {\rm if}\ \ep_i=-,\ \ep_{i+1}=+,\ i\neq r, \\
(\ep_1,\cdots,\ep_{r-1},\overset{r}{+}) & {\rm if}\ \ep_r=-,\ i=r, \\
0 & {\rm otherwise.}
\end{cases}
\end{equation}

Then the module $V_{{\rm sp}}^{(r)}$ is isomorphic to $V(\Lm_r)$ as a $\ge$-module.

\subsection{Type ${\rm D}_r$}\label{SectFundD}

Let $\frg=\frs\fro(2r,\mathbb{C})$ be the simple Lie algebra of type ${\rm D}_r$. The Cartan matrix $A=(a_{i,j})_{i,j\in I}$ of $\frg$ is given by
\[a_{i,j}=
\begin{cases}
2 & {\rm if}\ i=j, \\
-1 & {\rm if}\ |i-j|=1\ {\rm and}\ (i,j)\neq (r,r-1),\ (r-1,r),\ {\rm or}\ (i,j)=(r-2,r),\ (r,r-2), \\
0 & {\rm otherwise.}  
\end{cases}
\]

Define the partial order on the set $J_{{\rm D}}:=\{i,\ovl i|1\leq i\leq r\}$ by 
\begin{equation}\label{D-order}
 1< 2<\cd< r-1<\ ^{r}_{\ovl{r}}\ < \ovl{r-1}< \cd< \ovl
 2< \ovl 1.
\end{equation}
Note that there is no order between $r$ and $\ovl{r}$. For $\frg=\lan \frh,e_i,f_i(i\in I)\ran$, 
let us describe the vector representation 
$V(\Lm_1)$. Set ${\mathbf B}^{(r)}:=
\{v_i,v_{\ovl i}|i=1,2,\cd,r\}$ and define 
$V(\Lm_1):=\bigoplus_{v\in{\mathbf B}^{(r)}}\bbC v$. The weights of $v_i$, $v_{\ovl{i}}$ $(i=1,\cd,r)$ are as follows:
\begin{equation}\label{D-wtv}
 {\rm wt}(v_i)=\Lm_i-\Lm_{i-1},\q {\rm wt}(v_{\ovl{i}})=\Lm_{i-1}-\Lm_{i}\q (1\leq i\leq r-2,\ i=r),
\end{equation}
\[  {\rm wt}(v_{r-1})=\Lm_r+\Lm_{r-1}-\Lm_{r-2},\q {\rm wt}(v_{\ovl{r-1}})=\Lm_{r-2}-\Lm_{r-1}-\Lm_r, \]
where $\Lm_0=0$. We define the $\frg$-action on $V(\Lm_1)$ as follows:
\begin{eqnarray}
&& h v_j=\lan h,{\rm wt}(v_j)\ran v_j\ \ (h\in P^*,\ j\in J_{{\rm D}}), \\
&&f_iv_i=v_{i+1},\ f_iv_{\ovl{i+1}}=v_{\ovl i},\q
e_iv_{i+1}=v_i,\ e_iv_{\ovl i}=v_{\ovl{i+1}}
\q(1\leq i<r),\label{D-f1}\\
&&f_r v_r=v_{\ovl{r-1}},\ f_r v_{r-1}=v_{\ovl r},\qq 
e_r v_{\ovl r}=v_{r-1},\ e_r v_{\ovl{r-1}}=v_{r},\label{D-f2}
\end{eqnarray}
and the other actions are trivial.

Let $\Lm_i$ $(1\leq i\leq r-2)$ be the $i$-th fundamental weight of type ${\rm D}_r$. Similar to the ${\rm B}_r$ and ${\rm C}_r$ cases, the fundamental representation 
$V(\Lm_i)$
is embedded in $\wedge^i V(\Lm_1)$
with multiplicity free. In 
$\wedge^i V(\Lm_1)$, the highest (resp. lowest) weight 
vector $u_{\Lm_i}$ (resp. $v_{\Lm_i}$)
of $V(\Lm_i)$ is realized as the formula (\ref{C-h-l}).

The fundamental representations $V(\Lm_{r-1})$ and $V(\Lm_r)$ are also called the {\it spin representations}. They can be realized as follows: Set 
\[ {\mathbf B}_{{\rm sp}}^{(+,r)}\ ({\rm resp.}\ {\mathbf B}_{{\rm sp}}^{(-,r)}):=
\{(\ep_1,\cdots,\ep_r)|\ \ep_i\in\{+,- \},\q \ep_1\cdots\ep_r=+\ ({\rm resp.}\ -)\}, \]
\[ V_{{\rm sp}}^{(+,r)}\ ({\rm resp.}\ V_{{\rm sp}}^{(-,r)}):=\bigoplus_{v\in{\mathbf B}_{{\rm sp}}^{(+,r)}({\rm resp.}\ {\mathbf B}_{{\rm sp}}^{(-,r)})}\bbC v,\] 
and define the $\ge$-action on $V_{{\rm sp}}^{(\pm,r)}$ as follows:

\begin{equation}\label{Dsp-f0}
h_i(\ep_1,\cdots,\ep_r)=
\begin{cases}
\frac{\ep_i\cdot 1-\ep_{i+1}\cdot 1}{2}(\ep_1,\cdots,\ep_r) & {\rm if}\ i<r, \\
\frac{\ep_{r-1}\cdot 1+\ep_{r}\cdot 1}{2}\ep_r(\ep_1,\cdots,\ep_r) & {\rm if}\ i=r,
\end{cases}
\end{equation}

\begin{equation}\label{Dsp-f1}
f_i(\ep_1,\cdots,\ep_r)=
\begin{cases}
(\ep_1,\cdots,\overset{i}{-},\overset{i+1}{+},\cdots,\ep_r) & {\rm if}\ \ep_i=+,\ \ep_{i+1}=-,\ i\neq r, \\
(\ep_1,\cdots,\overset{r-1}{-},\overset{r}{-}) & {\rm if}\ \ep_{r-1}=+,\ \ep_r=+,\ i=r, \\
0 & {\rm otherwise,}
\end{cases}
\end{equation}

\begin{equation}\label{Dsp-f2}
e_i(\ep_1,\cdots,\ep_r)=
\begin{cases}

(\ep_1,\cdots,\overset{i}{+},\overset{i+1}{-},\cdots,\ep_r) & {\rm if}\ \ep_i=-,\ \ep_{i+1}=+,\ i\neq r, \\
(\ep_1,\cdots,\overset{r-1}{+},\overset{r}{+}) & {\rm if}\ \ep_{r-1}=-,\ \ep_r=-,\ i=r, \\
0 & {\rm otherwise.}
\end{cases}
\end{equation}

Then the module $V_{{\rm sp}}^{(+,r)}$ (resp. $V_{{\rm sp}}^{(-,r)}$) is isomorphic to $V(\Lm_r)$ (resp. $V(\Lm_{r-1})$) as a $\ge$-module.

\section{Factorization theorem}\label{DBCs}

In this section, we shall introduce (reduced) double Bruhat cells $G^{u,v}$, $L^{u,v}$, and their properties for $v=e$ and some special $u\in W$. In \cite{B-Z:2001, F-Z:1998}, these properties have been proven in more general setting. We shall state a relation between certain functions {\it generalized minors} on double Bruhat cells and crystal bases, which is our main result (Theorem \ref{thm1}). For $l\in \mathbb{Z}_{>0}$, we set $[1,l]:=\{1,2,\cdots,l\}$.

\subsection{Double Bruhat cells}\label{factpro}

Let $G$ be the simple complex algebraic group of classical type, $B$ and $B_-$ be two opposite Borel subgroups in $G$, $N\subset B$ and $N_-\subset B_-$ be their unipotent radicals, 
$H:=B\cap B_-$ a maximal torus. We set $\frg:={\rm Lie}(G)$ with the triangular decomposition $\frg=\frn_-\oplus \frh \oplus \frn$. Let $e_i$, $f_i$ $(i\in[1,r])$ be the generators of $\frn$, $\frn_-$. For $i\in[1,r]$ and $t \in \mathbb{C}$, we set
\begin{equation}\label{xiyidef} 
x_i(t):={\rm exp}(te_i),\ \ \ y_{i}(t):={\rm exp}(tf_i).
\end{equation}
Let $W:=\lan s_i |i=1,\cdots,r \ran$ be the Weyl group of $\frg$, where
$\{s_i\}$ are the simple reflections. We identify the Weyl group $W$ with ${\rm Norm}_G(H)/H$. An element 
\begin{equation}\label{smpl}
\ovl{s_i}:=x_i(-1)y_i(1)x_i(-1)
\end{equation}
is in ${\rm Norm}_G(H)$, which is representative of $s_i\in W={\rm Norm}_G(H)/H$ \cite{N1:2014}. For $u\in W$, let $u=s_{i_1}\cdots s_{i_n}$ be its reduced expression. Then we write $\ovl{u}=\ovl{s_{i_1}}\cdots \ovl{s_{i_n}}$, call $l(u):=n$ the length of $u$. We have two kinds of Bruhat decompositions of $G$ as follows:
\[ G=\displaystyle\coprod_{u \in W}B\ovl{u}B=\displaystyle\coprod_{u \in W}B_-\ovl{u}B_- .\]
Then, for $u$, $v\in W$, 
we define the {\it double Bruhat cell} $G^{u,v}$ as follows:
\[ G^{u,v}:=B\ovl{u}B \cap B_-\ovl{v}B_-. \]
This is biregularly isomorphic to a Zariski open subset of 
an affine space of dimension $r+l(u)+l(v)$ \cite[Theorem 1.1]{F-Z:1998}.

We also define the {\it reduced double Bruhat cell} $L^{u,v}$ as follows:
\[ L^{u,v}:=NuN \cap B_-vB_- \subset G^{u,v}. \] 
As is similar to the case $G^{u,v}$, $L^{u,v}$ is 
biregularly isomorphic to a Zariski open subset of an 
affine space of dimension $l(u)+l(v)$ \cite[Proposition 4.4]{B-Z:2001}.

\begin{defn}\label{redworddef}
Let $u=s_{i_1}\cdots s_{i_n}$ be a reduced expression of $u\in W$ $(i_1,\cdots,i_n\in [1,r])$. Then the finite sequence $\textbf{i}:=(i_1,\cdots,i_n)$ is called a {\it reduced word} for $u$.
\end{defn}

For example, the sequence $(1,2,3,1,2,1)$ is a reduced word of the longest element $s_1s_2s_3s_1s_2s_1$ of the Weyl group of type ${\rm A}_3$. For all the cases ${\rm A}_r$, ${\rm B}_r$, ${\rm C}_r$ and ${\rm D}_r$, we fix the reduced word $\textbf{i}_0$ of the longest element as follows:
\begin{equation}\label{redwords}
\textbf{i}_0=
\begin{cases}
(1,2,\cdots,r,1,2,\cdots,r-1,\cdots,1,2,3,1,2,1) & {\rm for}\ {\rm A}_r,\\
(1,2,\cdots,r-1,r)^r & {\rm for}\ {\rm B}_r,\ {\rm C}_r,\\
(1,2,\cdots,r-1,r)^{r-1} & {\rm for}\ {\rm D}_r.
\end{cases}
\end{equation}
In this paper, we mainly treat (reduced) Double Bruhat cells of the form $G^{u,e}:=BuB \cap B_-$, $L^{u,e}:=NuN \cap B_-$, and the element $u\in W$ whose reduced word can be written as a left factor of $\textbf{i}_0$.

\subsection{Factorization theorem}\label{fuctorisec}

In this subsection, we shall introduce the isomorphisms between double Bruhat cell $G^{u,e}$ and $H\times (\mathbb{C}^{\times})^{l(u)}$, and between $L^{u,e}$ and $(\mathbb{C}^{\times})^{l(u)}$. As in the previous subsection, let $G$ be a complex classical algebraic group of type ${\rm A}_r$, ${\rm B}_r$, ${\rm C}_r$ and ${\rm D}_r$. For a reduced word $\textbf{i}=(i_1, \cdots ,i_n)$ 
($i_1,\cdots,i_n\in[1,r]$), 
we define a map $x^G_{\textbf{i}}:H\times \mathbb{C}^n \rightarrow G$ as 
\begin{equation}\label{xgdef}
x^G_{\textbf{i}}(a; t_1, \cdots, t_n):=a\cdot y_{i_1}(t_1)\cdots y_{i_n}(t_n).
\end{equation}

\begin{thm}\label{fp}{\cite{F-Z:1998}} For $u\in W$ and its reduced word ${\rm \bf{i}}$, the map $x^G_{{\rm \bf{i}}}$ is a biregular isomorphism from $H\times (\mathbb{C}^{\times})^{l(u)}$ to a Zariski open subset of $G^{u,e}$. 
\end{thm}

Next, for $i \in [1,r]$ and $t\in \mathbb{C}^{\times}$, we define as follows:
\begin{equation}\label{alxmdef}
\alpha_i^{\vee}(t):=t^{h_i},\ \ 
 x_{-i}(t):=y_{i}(t)\alpha_i^{\vee}(t^{-1}).
\end{equation}
For $\textbf{i}=(i_1, \cdots ,i_n)$
($i_1,\cdots,i_n\in[1,r]$), 
we define a map $x^L_{\textbf{i}}:\mathbb{C}^n \rightarrow G$ as 
\begin{equation}\label{xldef}
x^L_{\textbf{i}}(t_1, \cdots, t_n):=x_{-i_1}(t_1)\cdots x_{-i_n}(t_n).
\end{equation}
We have the following theorem which is similar to the previous one.

\begin{thm}\label{fp2}{\cite{B-Z:2001}}
For $u\in W$ and its reduced word ${\rm \bf{i}}$, the map $x^L_{{\rm \bf{i}}}$ is a biregular isomorphism from $ (\mathbb{C}^{\times})^{l(u)}$ to a Zariski open subset of $L^{u,e}$. 
\end{thm}

We define a map
$\ovl{x}^G_{\textbf{i}}:H\times(\mathbb{C}^{\times})^{n}\rightarrow
G^{u,e}$ as
\[ \ovl{x}^G_{\textbf{i}}(a;t_1,\cdots,t_n)
=ax^L_{{\rm \bf{i}}}(t_1,\cdots,t_n), \]
where $a\in H$ and $(t_1,\cdots,t_n)\in (\mathbb{C}^{\times})^{n}$. In \cite{KaN:2015, KaN2:2016, Ka1:2016}, we have proven the following proposition in the case $G$ is of type ${\rm A}_r$, ${\rm B}_r$ and ${\rm C}_r$. Similarly, we can prove it in the case $G$ is type ${\rm D}_r$. 

\begin{prop}\label{gprime}
In the above setting, the map $\ovl{x}^G_{{\rm \bf{i}}}$ is a biregular isomorphism between $H\times(\mathbb{C}^{\times})^{n}$ and a Zariski open subset of $G^{u,e}$.
\end{prop}

\section{Cluster algebras and generalized minors}\label{CluSect}
Following \cite{A-F-Z:2005,F-Z:1998,FZ2:2002,M-M-A:2010}, we review the definitions of cluster algebras and their generators called cluster variables. It is known that the coordinate rings of double Bruhat cells have cluster algebra structures, and generalized minors are their cluster variables \cite{A-F-Z:2005, GY:2016}. We will refer to a relation between certain cluster variables on double Bruhat cells and crystal bases in Sect.\ref{gmc}.

We set $[1,l]:=\{1,2,\cdots,l\}$ and $[-1,-l]:=\{-1,-2,\cdots,-l\}$ for $l\in \mathbb{Z}_{>0}$. For $n,m\in \mathbb{Z}_{>0}$, let $x_1, \cdots ,x_n,x_{n+1}, \cdots
,x_{n+m}$ be commuting variables and $\mathcal{P}$ be a free multiplicative
abelian group generated by $x_{n+1},\cdots,x_{n+m}$. We set ${\mathbb
Z}\cP:={\mathbb Z}[x_{n+1}^{\pm1}, \cdots ,x_{n+m}^{\pm1}]$. Let $\cF:=\mathbb{C}(x_{1}, \cdots ,x_{n},x_{n+1},\cdots,x_{n+m})$ 
be the field of rational functions.

\subsection{Cluster algebras of geometric type}

In this subsection, we recall the definitions of cluster algebras. Let $\tilde{B}=(b_{ij})_{1\leq i\leq n+m,\ 1\leq j \leq n}$ be an $(n+m)\times
n$ integer matrix. The {\it principal part} $B$ of $\tilde{B}$ is obtained from $\tilde{B}$ by deleting the last $m$ columns. For $\tilde{B}$ and $k\in [1,n]$, the new $(n+m)\times n$ integer matrix $\mu_k(\tilde{B})=(b'_{ij})$ is defined by
\[b_{ij}':=
\begin{cases}
	-b_{ij} & {\rm if}\ i=k\ {\rm or}\ j=k, \\
	b_{ij}+\frac{|b_{ik}|b_{kj}+b_{ik}|b_{kj}|}{2} & {\rm otherwise}.
\end{cases}
\]
One calls $\mu_k(\tilde{B})$ the {\it matrix mutation} in direction $k$ of $\tilde{B}$. If there exists a positive 
integer diagonal matrix $D$ such that $DB$ is skew symmetric, we say $B$ is {\it skew symmetrizable}. It is easily verified that if $\tilde{B}$ has a skew symmetrizable principal part then $\mu_k(\tilde{B})$ also has a skew symmetrizable principal part {\cite[Proposition\ 3.6]{M-M-A:2010}}. 
We can also verify that $\mu_k\mu_k(\tilde{B})=\tilde{B}$. Define $\textbf{x}:=(x_1,\cdots,x_{n+m})$ and we call the pair $(\textbf{x}, \tilde{B})$ {\it initial seed}. For $1\leq k\leq n$, a new cluster variable $x_k'$ is defined by
\[ x_k x_k' = 
\prod_{1\leq i \leq n+m,\ b_{ik}>0} x_i^{b_{ik}}
+\prod_{1\leq i \leq n+m,\ b_{ik}<0} x_i^{-b_{ik}}. \]
Let $\mu_k(\textbf{x})$ be the set of variables obtained from $\textbf{x}$ by replacing $x_k$ by $x'_k$. Ones call the pair $(\mu_k(\textbf{x}), \mu_k(\tilde{B}))$ the {\it mutation} in direction $k$ of the seed $(\textbf{x}, \tilde{B})$.

Now, we can repeat this process of mutation and obtain a set of 
seeds inductively. 
Hence, each seed consists of an $(n+m)$-tuple of variables and a matrix. 
Ones call this $(n+m)$-tuple and matrix 
{\it cluster} and {\it exchange matrix} respectively. 
Variables in cluster are called {\it cluster variables}.

\begin{defn}{\cite{F-Z:1998, M-M-A:2010}}\label{clusterdef}
Let $\tilde{B}$ be an integer matrix whose principal part is 
skew symmetrizable and $\Sigma=(\textbf{x},\tilde{B})$ a seed. 
We set ${\mathbb A}:={\mathbb Z}\cP$. 
The cluster algebra (of geometric type) 
$\cA=\cA(\Sigma)$ over $\mathbb A$ associated with
seed $\Sigma$ is defined as an ${\mathbb A}$-subalgebra of 
$\cF$ generated by all cluster variables in all 
seeds which can be obtained from $\Sigma$ by sequences of mutations.
\end{defn}



\subsection{Cluster algebra ${\cA}({\rm \bf{i}})$}
\label{UCACA}

Let $G$ be a simple classical algebraic group, 
$\ge:={\rm Lie}(G)$ and $A=(a_{i,j})$ be its Cartan matrix. 
In Definition \ref{redworddef}, 
we define a reduced word ${\rm \bf{i}}=(i_1,\cdots,i_{l(u)})$ 
for an element $u$ of Weyl group $W$. 
In this subsection, we define the cluster algebra ${\cA}({\rm \bf{i}})$, 
which obtained from ${\rm \bf{i}}$.
It satisfies that ${\cA}({\rm \bf{i}})\otimes \mathbb{C}$ is 
isomorphic to the coordinate ring $\mathbb{C}[G^{u,e}]$ 
of the double Bruhat cell \cite{A-F-Z:2005,GY:2016}. Indeed, in \cite{A-F-Z:2005}, it is shown that
$\mathbb{C}[G^{u,e}]$ holds the structure of an upper cluster algebra and 
in \cite{GY:2016} it possesses the structure of a cluster algebra.
Let $i_k$ $(k\in[1,l(u)])$ be the $k$-th index of ${\rm \bf{i}}$ from the left. For $t\in [-1,-r]$, we set $i_t:=t$.

For $k\in[-1,-r]\cup[1,l(u)]$, 
we denote by $k^+$ the smallest index $l$ such that $k<l$ and $|i_l|=|i_k|$. 
For example, if ${\rm \bf{i}}=(1,2,3,1,2)$ then, 
$1^+=4$, $2^+=5$ and $3^+$ is not defined. 
We define a set e({\rm \bf{i}}) as 
\[ e({\rm \bf{i}}):=\{k\in[1,l(u)]| k^+\ {\rm is\ well}-{\rm defined}\}.  
\]
Following \cite{A-F-Z:2005}, 
we define a quiver $\Gamma_{{\rm \bf{i}}}$ as follows. 
The vertices of $\Gamma_{{\rm \bf{i}}}$ are the numbers $[-1,-r]\cup[1,l(u)]$. 
For two vertices $k\in [-1,-r]\cup[1,l(u)]$ and $l\in[1,l(u)]$ with $k<l$, 
there exists an arrow $k\rightarrow l$ (resp. $l\rightarrow k$) 
if and only if $l=k^+$ (resp. $l<k^+<l^+$ and $a_{i_k,i_l}<0$). 
For $k,\ l\in [-1,-r]$, there exists an arrow $k\rightarrow l$ 
if and only if $l<l^+<k^+$ and $a_{i_k,i_l}<0$. 
Next, let us define a matrix $\tilde{B}=\tilde{B}({\rm \bf{i}})$. 
\begin{defn}
Let $\tilde{B}({\rm \bf{i}})$ be an integer matrix with 
rows labelled by all the indices in $[-1,-r]\cup [1,l(u)]$ 
and columns labelled by all the indices in $e({\rm \bf{i}})$. 
For $k\in[-1,-r]\cup [1,l(u)]$ and $l\in e({\rm \bf{i}})$, 
an entry $b_{kl}$ of $\tilde{B}({\rm \bf{i}})$ is determined as follows: 
If there exists an arrow $k\rightarrow l$ 
(resp. $l\rightarrow k$) in $\Gamma_{{\rm \bf{i}}}$, then
\[
b_{kl}:=\begin{cases}
		1\ ({\rm resp.}\ -1)& {\rm if}\ |i_k|=|i_l|, \\
		-a_{|i_k||i_l|}\ ({\rm resp.}\ a_{|i_k||i_l|})& {\rm if}\ |i_k|\neq|i_l|.
	\end{cases}
\]
If there exist no arrows between $k$ and $l$, we set $b_{kl}=0$.
\end{defn}

\begin{prop}\label{propss}{\cite[Proposition\ 2.6]{A-F-Z:2005}}
$\tilde{B}({\rm \bf{i}})$ is skew symmetrizable. 
\end{prop}

\begin{defn}
By Proposition \ref{propss}, we can construct a cluster algebra from the 
matrix $\tilde{B}({\rm \bf{i}})$ by applying mutations. 
We denote this cluster algebra by $\cA({\rm \bf{i}})$.
\end{defn}

\subsection{Generalized minors and bilinear form}\label{bilingen}

Set ${\cA}({\rm \bf{i}})_{\mathbb{C}}:={\cA}({\rm \bf{i}})\otimes \mathbb{C}$. 
It is known that the coordinate ring $\mathbb{C}[G^{u,e}]$ 
of the double Bruhat cell is isomorphic to ${\cA}({\rm \bf{i}})_{\mathbb{C}}$ 
(Theorem \ref{clmainthm}). 
To describe this isomorphism explicitly, we need generalized minors.  

We set $G_0:=N_-HN$, and let $x=[x]_-[x]_0[x]_+$ with $[x]_-\in N_-$, $[x]_0\in H$, $[x]_+\in N$ be the corresponding decomposition. 

\begin{defn}
For $i\in[1,r]$ and $u'\in W$, the generalized minor $\Delta_{u\Lambda_i,\Lambda_i}$ is a regular function on $G$ whose restriction to the open set $u G_0$ is given by $\Delta_{u\Lambda_i,\Lambda_i}(x)=([u^{-1}x ]_0)^{\Lambda_i}$. Here, $\Lambda_i$ is the $i$-th  fundamental weight, for $a=t^h\in H$ $(h\in P^{*} )$ and $\lambda\in P$, we set $a^{\lambda}:=t^{\lambda(h)}$. In particular, we write $\Delta_{\Lambda_i}:=\Delta_{\Lambda_i,\Lambda_i}$ and call it a {\it principal minor}.
\end{defn}

We can calculate generalized minors by using a bilinear form in the fundamental representation of $\ge={\rm Lie}(G)$ (Sect.\ref{SectFund}). Let $x_i(t)$, $y_i(t)$ be the ones in Sect.\ref{DBCs} (\ref{xiyidef}), and $\omega:\ge\to\ge$ be the anti-involution 
\[
\omega(e_i)=f_i,\q
\omega(f_i)=e_i,\q \omega(h)=h,
\] and extend it to $G$ by setting
$\omega(x_i(c))=y_{i}(c)$, $\omega(y_{i}(c))=x_i(c)$ and $\omega(t)=t$
$(t\in H)$. There exists a $\ge$ (or $G$)-invariant bilinear form on the
fundamental representation $V(\Lm_i)$ of $\ge$ such that 
\[
 \lan au,v\ran=\lan u,\omega(a)v\ran,
\q\q(u,v\in V(\lm),\,\, a\in \ge\ (\text{or }G)).
\]
For $g\in G$, 
we have the following simple fact:
\[
 \Del_{\Lm_i}(g)=\lan gu_{\Lm_i},u_{\Lm_i}\ran,
\]
where $u_{\Lm_i}$ is a properly normalized highest weight vector in
$V(\Lm_i)$. Hence, for $w\in W$, we have
\begin{equation}\label{minor-bilin}
 \Del_{w\Lm_i,\Lm_i}(g)=
\Del_{\Lm_i}({\ovl w}^{-1}g)=
\lan {\ovl w}^{-1}g\cdot u_{\Lm_i},u_{\Lm_i}\ran
=\lan g\cdot u_{\Lm_i}\, ,\, \ovl{w}\cdot u_{\Lm_i}\ran,
\end{equation}
where $\ovl w$ is the one we defined in Sect.\ref{DBCs} (\ref{smpl}), and note that $\omega(\ovl s_i^{\pm})=\ovl s_i^{\mp}$.

\subsection{Cluster algebras on Double Bruhat cells}
\label{CADBC}

For $u=s_{i_1}s_{i_2}\cdots s_{i_n}$ and $k\in [1,l(u)]$, we set
\begin{equation}\label{inc}
u_{\leq k}=u_{\leq k}({\rm \bf{i}}):=s_{i_1}s_{i_2}\cdots s_{i_k}.
\end{equation}
For $k \in [-1,-r]$, we set $u_{\leq k}:=e$. For $k \in [-1,-r]\cup [1,\ l(u)]$, we define $\Delta(k;{\rm \bf{i}})(x):=\Delta_{u_{\leq k} \Lambda_{|i_k|},\Lambda_{|i_k|}}(x)$. 

We set $F({\rm \bf{i}}):=\{ \Delta(k;{\rm \bf{i}})(x)|k \in [-1,-r]\cup[1,\ l(u)] \}$. It is known that the set $F({\rm \bf{i}})$ is an algebraically independent generating set for the field of rational functions $\mathbb{C}(G^{u,e})$ \cite[Theorem 1.12]{F-Z:1998}. Then, we have the following.

\begin{thm}\label{clmainthm}{\cite{A-F-Z:2005,GY:2016}}
The isomorphism of fields 
$\varphi :\mathcal{F} \rightarrow \mathbb{C}(G^{u,e})$ 
defined by $\varphi (x_k)=\Delta(k;{\rm \bf{i}})
\ (k \in [-1,-r]\cup [1,l(u)] )$ 
restricts to an isomorphism of algebras 
${\cA}({\rm \bf{i}})_{\mathbb{C}}\rightarrow \mathbb{C}[G^{u,e}]$.
\end{thm}

\section{Monomial realizations and Demazure crystals}

As mentioned in the beginnings of Sect.\ref{DBCs} and \ref{CluSect}, 
our findings are relations between generalized minors 
on double Bruhat cells and crystal bases. 
More precisely, we shall describe generalized minors in terms of 
the {\it monomial realizations} of Demazure crystals. 
Let us recall these notion in this section. 
Let $\ge$ be a complex simple Lie algebra 
and we will use the same notation as in Sect. \ref{SectFund}.

\subsection{Monomial realizations of crystals}

In this subsection, we review the monomial realizations of 
crystals~\cite{K2:1991,K:2003,Nj:2003}. First, let us recall the crystals.

\begin{defn}\label{defcry}\cite{K0:1990}
A~{\it crystal} associated with the Cartan matrix~$A$ is a~set~$B$ 
together with the maps $\text{wt}: B \rightarrow P$,
$\tilde{e_{i}}$, $\tilde{f_{i}}: B \cup \{0\} 
\rightarrow B \cup \{0\}$ and $\varepsilon_i$,
$\varphi_i: B \rightarrow {\mathbb Z} \cup \{-\infty \}$, 
$i \in I$, satisfying some properties.
\end{defn}
We call $\{\tilde{e}_i,\ \tilde{f}_i\}_{i\in I}$ the {\it Kashiwara operators}. 
Let $U_q(\mathfrak g)$ be the quantum enveloping algebra 
\cite{K0:1990} associated with the Cartan matrix~$A$, that is,  
$U_q(\mathfrak g)$ has generators $\{e_i,\ f_i,\ h_i |\ i\in I\}$ over 
$\mathbb{C}(q)$ satisfying some relations, 
where $q$ is an indeterminate. Let $V(\lm)$ 
($\lm\in P^+=\oplus_{i\in I}\mathbb{Z}_{\geq0}\Lambda_i$) 
be the finite dimensional irreducible representation of 
$U_q(\mathfrak g)$ which has the highest weight vector $u_{\lm}$, 
and $(L(\lm),B(\lm))$ be the crystal base of $V(\lm)$. 
The crystal $B(\lm)$ has a crystal structure.

Let us introduce monomial realizations of crystals which realize each element of 
crystals as a~certain Laurent monomial. 
We define a~set of integers $p=(p_{j,i})_{j,i \in I,\; j \neq i}$ such that
\begin{gather*}
p_{j,i}=
\begin{cases}
1 & \text{if} \quad  j<i,
\\
0 & \text{if} \quad  i<j.
\end{cases}
\end{gather*}
For doubly-indexed variables 
$\{Y_{s,i} \,|\, i \in I$, $s\in \mathbb{Z}\}$, 
we def\/ine the set of monomials
\begin{gather*}
{\mathcal Y}:=\left\{Y=\prod\limits_{s \in \mathbb{Z},\ i \in I}
Y_{s,i}^{\zeta_{s,i}}\, \Bigg| \,\zeta_{s,i} \in \mathbb{Z},\
\zeta_{s,i} =0~\text{except for f\/initely many}~(s,i) \right\}.
\end{gather*}

Finally, we def\/ine maps $\text{wt}: {\mathcal Y} \rightarrow P$, $\varepsilon_i$, $\varphi_i: {\mathcal Y} \rightarrow
\mathbb{Z}$, $i \in I$. 
For $Y=\prod\limits_{s \in \mathbb{Z},\; i \in I} Y_{s,i}^{\zeta_{s,i}}\in {\mathcal Y}$, 
let
\begin{gather}\label{wtph}
\text{wt}(Y):= \sum\limits_{i,s}\zeta_{s,i}\Lambda_i,\!
\quad
\varphi_i(Y):=\max\left\{\! \sum\limits_{k\leq s}\zeta_{k,i}  \,|\, s\in \mathbb{Z} \!\right\},\!
\quad
\varepsilon_i(Y):=\varphi_i(Y)-\text{wt}(Y)(h_i).
\end{gather}
We set
\begin{gather}
\label{asidef}
A_{s,i}:=Y_{s,i}Y_{s+1,i}\prod\limits_{j\neq i}Y_{s+p_{j,i},j}^{a_{j,i}}
\end{gather}
and def\/ine the Kashiwara operators as follows
\begin{gather*}
\tilde{f}_iY=
\begin{cases}
A_{n_{f_i},i}^{-1}Y & \text{if} \quad  \varphi_i(Y)>0,
\\
0 & \text{if} \quad  \varphi_i(Y)=0,
\end{cases}
\qquad
\tilde{e}_iY=
\begin{cases}
A_{n_{e_i},i}Y & \text{if} \quad  \varepsilon_i(Y)>0,
\\
0 & \text{if} \quad  \varepsilon_i(Y)=0,
\end{cases}
\end{gather*}
where
\begin{gather*}
n_{f_i}:=\min \left\{n \,\Bigg|\, \varphi_i(Y)= \sum\limits_{k\leq n}\zeta_{k,i}\right\},
\qquad
n_{e_i}:=\max \left\{n \,\Bigg|\, \varphi_i(Y)= \sum\limits_{k\leq n}\zeta_{k,i}\right\}.
\end{gather*}

Then the following theorem holds:


\begin{thm}[\cite{K:2003,Nj:2003}]\label{monorealmain}\quad

\begin{itemize}
\item[(i)] For the set $p=(p_{j,i})$ as above, 
$({\mathcal Y}, \text{wt}, \varphi_i, \varepsilon_i,\tilde{f}_i,
\tilde{e}_i)_{i\in I}$ is a~crystal.
When we emphasize~$p$, we write ${\mathcal Y}$ as ${\mathcal Y}(p)$.
\item[(ii)] If a~monomial $Y \in {\mathcal Y}(p)$ 
satisfies $\varepsilon_i(Y)=0$ $($resp.~$\varphi_i(Y)=0)$ 
for all $i \in I$,
then the connected component containing~$Y$ is isomorphic to $B(\text{wt}(Y))$ 
$($resp.\ $B(w_0\cdot\text{wt}(Y))$, where $w_0$ is the longest element of $W)$.
\end{itemize}

\end{thm}

\begin{ex}\label{monorealex1}
Let us consider the case of type ${\rm C}_2$. By $(\ref{asidef})$, we get
\[ A_{s,1}=\frac{Y_{s,1}Y_{s+1,1}}{Y_{s,2}},\q  A_{s,2}
=\frac{Y_{s,2}Y_{s+1,2}}{Y^2_{s+1,1}}\qq (s\in\mathbb{Z}). \]
We set $Y:=Y_{0,2}\in \cY$. Following the definitions, 
we obtain {\rm wt}$(Y)=\Lm_2$, $\varphi_2(Y)=1$, 
$\varepsilon_2(Y)=\varphi_2(Y)-\Lm_2(h_2)=0$ and 
$\varphi_1(Y)=\varepsilon_1(Y)=0$. It follows from $n_{f_2}=0$ that
\[ \tilde{f}_2 Y=A^{-1}_{0,2}Y_{0,2}
=\frac{Y^2_{1,1}}{Y_{1,2}},\q \tilde{f}_1 Y=0.  \]
Similar to this, we also obtain
\[ \tilde{f}_1 \tilde{f}_2 Y=A^{-1}_{1,1}\frac{Y^2_{1,1}}{Y_{1,2}}
=\frac{Y_{1,1}}{Y_{2,1}},\q \tilde{f}_2 \tilde{f}_2 Y=0.  \]
Repeating this argument, we see that the connected component 
containing $Y=Y_{0,2}$ is as follows:
\begin{equation}\label{monorealdia1}
Y_{0,2}\overset{2}{\longrightarrow}\frac{Y^2_{1,1}}{Y_{1,2}}
\overset{1}{\longrightarrow}\frac{Y_{1,1}}{Y_{2,1}} 
\overset{1}{\longrightarrow}\frac{Y_{1,2}}{Y^2_{2,1}}
\overset{2}{\longrightarrow}\frac{1}{Y_{2,2}}.
\end{equation}
The graph $($\ref{monorealdia1}$)$ is
the monomial realization of the crystal base 
$B(\Lm_2)=B({\rm wt}(Y_{0,2}))$.
\end{ex}

\begin{ex}\label{monorealexspin1}
Let us consider the case of type ${\rm B}_3$. 
By $(\ref{asidef})$, we get
\[ A_{s,1}=\frac{Y_{s,1}Y_{s+1,1}}{Y_{s,2}},\q  
A_{s,2}=\frac{Y_{s,2}Y_{s+1,2}}{Y_{s+1,1}Y^2_{s,3}},\q 
A_{s,3}=\frac{Y_{s,3}Y_{s+1,3}}{Y_{s+1,2}},
\qq, \]
for $s\in\mathbb{Z}$.

We set $Y:=Y_{0,3}\in \cY$. Following the definitions, 
we obtain {\rm wt}$(Y)=\Lm_3$, $\varphi_3(Y)=1$, 
$\varepsilon_3(Y)=\varphi_3(Y)-\Lm_3(h_3)=0$ 
and $\varphi_i(Y)=\varepsilon_i(Y)=0$ $(i=1,2)$. It follows from $n_{f_3}=0$ that
\[ \tilde{f}_3 Y=A^{-1}_{0,3}Y_{0,3}=\frac{Y_{1,2}}{Y_{1,3}},\q \tilde{f}_i Y=0\ (i=1,2).  \]
Similar to this, we also obtain
\[ \tilde{f}_2 \tilde{f}_3 Y=\frac{Y_{1,3}Y_{2,1}}{Y_{2,2}},\q 
\tilde{f}_1\tilde{f}_2 \tilde{f}_3 Y=\frac{Y_{1,3}}{Y_{3,1}},\q 
\tilde{f}_3\tilde{f}_2 \tilde{f}_3 Y=\frac{Y_{2,1}}{Y_{2,3}}.
  \]
Repeating this argument, we see that 
the connected component containing $Y=Y_{0,3}$ is as follows:
\begin{equation}\label{monorealspindia1}
Y_{0,3}\overset{3}{\longrightarrow}\frac{Y_{1,2}}{Y_{1,3}}
\overset{2}{\longrightarrow}\frac{Y_{1,3}Y_{2,1}}{Y_{2,2}} 
\overset{1}{\longrightarrow}\frac{Y_{1,3}}{Y_{3,1}}
\qq \qq \qq \qq
\end{equation}
\[ \qq \downarrow\ ^3 \qq \q \downarrow\ ^3 \]
\[ \qq \qq \qq \qq \qq \frac{Y_{2,1}}{Y_{2,3}}
\overset{1}{\longrightarrow}\frac{Y_{2,2}}{Y_{2,3}Y_{3,1}}
\overset{2}{\longrightarrow} \frac{Y_{2,3}}{Y_{3,2}}
\overset{3}{\longrightarrow}\frac{1}{Y_{3,3}}.  \]
The graph $($\ref{monorealspindia1}$)$ is
the monomial realization of the crystal base 
$B(\Lm_3)=B({\rm wt}(Y_{0,3}))$.
\end{ex}

\subsection{Demazure crystals}\label{demazure}

For $w\in W$ and $\lambda\in P^+$, an {\it upper Demazure crystal}
$B^{+}(\lambda)_w\subset B(\lambda)$ is inductively defined as follows. 

\begin{defn}
Let $u_{\lambda}$ be the highest weight vector of $B(\lambda)$. For
 the identity element $e$ of $W$, we set
 $B^{+}(\lambda)_e:=\{u_{\lambda}\}$. 
For $w\in W$, if $s_iw<w$, 
\[ B^{+}(\lambda)_w:=\{\tilde{f}_i^kb\ |\ k\geq0,\ b\in
 B^{+}(\lambda)_{s_iw},\ \tilde{e}_ib=0\}\setminus \{0\}. 
\]
\end{defn}
Similarly, we define a {\it lower Demazure crystal}
$B^{-}(\lambda)_w$ inductively.
\begin{defn}
Let $v_{\lambda}$ be the lowest weight vector of $B(\lambda)$. 
We set $B^{-}(\lambda)_e:=\{v_{\lambda}\}$. For $w\in W$, 
if $s_iw<w$, 
\[ B^{-}(\lambda)_w:=\{\tilde{e}_i^kb\ |\ k\geq0,\ b\in
 B^{-}(\lambda)_{s_iw},\ \tilde{f}_ib=0\}\setminus \{0\}. 
\]
\end{defn}

\begin{thm}\label{kashidem}{\cite{K3:2002}}
For $w\in W$, let $w=s_{i_1}\cdots s_{i_n}$ be an arbitrary reduced
 expression. Let $u_{\lambda}$ $(resp.\ v_{\lambda})$ be the highest
 $(resp.\ lowest)$ weight vector of $B(\lambda)$. Then 
\begin{eqnarray*} 
&&B^{+}(\lambda)_w=\{\tilde{f}_{i_1}^{a(1)}\cdots
\tilde{f}_{i_n}^{a(n)}u_{\lambda}|a(1),\cdots,a(n)\in\mathbb{Z}_{\geq0}
\}
\setminus\{0\}, \\
&& B^{-}(\lambda)_w=\{\tilde{e}_{i_1}^{a(1)}\cdots
\tilde{e}_{i_n}^{a(n)}v_{\lambda}|a(1),\cdots,a(n)\in\mathbb{Z}_{\geq0}
 \}\setminus\{0\}. 
\end{eqnarray*}
\end{thm}

\section{Generalized minors and crystals}\label{gmc}

Let $G$ be a complex classical algebraic group. In this section, we describe certain initial cluster variables on $G^{u,e}$ by monomial realizations of Demazure crystals. In the rest of paper, we only treat elements $u\in W$ $(l(u)=n)$ whose reduced word ${\rm \bf{i}}$ can be written as a left factor of ${\rm \bf{i}}_0$ in (\ref{redwords}), which means that ${\rm \bf{i}}$ is defined by
\begin{equation}\label{redwords2}
\begin{cases}
(\underbrace{1,2,\cdots,r}_{1{\rm st\ cycle}},\underbrace{1,2,\cdots,r-1}_{2{\rm nd\ cycle} },\cdots,\underbrace{1,2,\cdots,r-m+1}_{m-1{\rm th\ cycle}},\underbrace{1,2,\cdots,i_n}_{m{\rm th\ cycle}}) & {\rm for}\ {\rm A}_r,\\
(\underbrace{1,2,\cdots,r}_{1{\rm st\ cycle}},\underbrace{1,2,\cdots,r}_{2{\rm nd\ cycle}}\cdots \underbrace{1,2,\cdots,r}_{m-1{\rm th\ cycle}},\underbrace{1,2,\cdots,i_n}_{m{\rm th\ cycle}}) & {\rm for}\ {\rm B}_r,\ {\rm C}_r,\ {\rm D}_r,
\end{cases}
\end{equation}
where $m\in [1,r]$ (for type ${\rm A}_r,\ {\rm B}_r,\ {\rm C}_r$) and $m\in [1,r-1]$ (for type ${\rm D}_r$), $i_n\in [1,r-m]$ (for type ${\rm A}_r$) and $i_n\in [1,r]$ (for type ${\rm B}_r,\ {\rm C}_r,\ {\rm D}_r$). Let $i_k$ be the $k$-th index of  ${\rm \bf{i}}$ from the left, and belong to $m'$-th cycle ($m'\leq m$). As we shall show in Lemma \ref{gmlem}, we may assume $i_n=i_k$. 

By Theorem \ref{clmainthm}, we can regard $\mathbb{C}[G^{u,e}]$ as a cluster algebra and $\{\Delta(k;{\rm \bf{i}})\}$ as its cluster
variables. Each $\Delta(k;{\rm \bf{i}})$ is a regular function on
$G^{u,e}$. On the other hand, by Proposition \ref{gprime} (resp. Theorem
\ref{fp2}), $\Delta(k;{\rm \bf{i}})$ can be seen as a function on
$H\times
(\mathbb{C}^{\times})^{l(u)}$ (resp. $(\mathbb{C}^{\times})^{l(u)}$). Here,
we change the variables of 
$\{\Delta(k;{\rm \bf{i}})\}$ as follows: 
\begin{defn}\label{gendef}
Along (\ref{redwords2}), we set a variable $\textbf{Y}\in  (\mathbb{C}^{\times})^{n}$ as
\begin{equation}\label{yset}
\textbf{Y}:= \\
\begin{cases}
(Y_{1,1},Y_{1,2},\cdots,Y_{1,r},Y_{2,1},Y_{2,2},\cdots,Y_{2,r-1}, \\
\qq \qq \qq \cdots,Y_{m-1,1},\cdots,Y_{m-1,r-m+1},Y_{m,1},\cdots,Y_{m,i_n}), & ({\rm A}_r) 
\\
(Y_{1,1},Y_{1,2},\cdots,Y_{1,r},Y_{2,1},Y_{2,2},\cdots,Y_{2,r}, \\
\qq \qq \qq \cdots,Y_{m-1,1},\cdots,Y_{m-1,r},Y_{m,1},\cdots,Y_{m,i_n}), & ({\rm B}_r,\ {\rm C}_r\ {\rm D}_r)
\end{cases}
\end{equation}
and for $a\in H$, define
\[ 
\Delta^G(k;{\rm \bf{i}})(a,\textbf{Y}):=(\Delta(k;{\rm \bf{i}})\circ
 \ovl{x}^G_{{\rm \bf{i}}})(a,\textbf{Y}), \q
\Delta^L(k;{\rm \bf{i}})(\textbf{Y}):=(\Delta(k;{\rm \bf{i}})\circ
 x^L_{{\rm \bf{i}}})(\textbf{Y}),
\]
where $\ovl{x}^G_{{\rm \bf{i}}}$ and $x^L_{{\rm \bf{i}}}$ are as in Subsection \ref{fuctorisec}.
\end{defn}

\begin{rem}\label{importantrem}
If we see the variables $Y_{s,0}$, $Y_{s,r+1}$ 
$(1\leq s\leq m)$ then 
we understand $Y_{s,0}=Y_{s,r+1}=1$. For example, if $i=1$ then $Y_{s,i-1}=1$.
\end{rem}

Next proposition implies that
$\Delta^G(k;{\rm \bf{i}})(a,\textbf{Y})$ is 
immediately obtained from $\Delta^L(k;{\rm \bf{i}})(\textbf{Y})$:
\begin{prop}\cite{KaN:2015, KaN2:2016, Ka1:2016}\label{gprop}
We set $d:=i_k$. For $a\in H$, we have
\[ 
\Delta^G(k;{\rm \bf{i}})(a,\textbf{Y})=
(a^{u_{\leq k}\Lambda_d})\Delta^L(k;{\rm \bf{i}})(\textbf{Y}) 
\] 
\end{prop}

In the rest of the paper, we will treat $\Delta^L(k;{\rm \bf{i}})(\textbf{Y})$ only due to this proposition.

\begin{lem}\cite{KaN:2015, KaN2:2016, Ka1:2016}\label{gmlem}
Let ${\rm \bf{i}}$, $\textbf{Y}$ be as in $(\ref{redwords2})$, $(\ref{yset})$, and $u\in W$ be an element whose reduced word is ${\rm \bf{i}}$. Let $i_{n+1}\in [1,r]$ be an index such that $u':=us_{i_{n+1}}\in W$ satisfies $l(u')>l(u)$. We set the reduced word ${\rm \bf{i}}'$ for $u'$ as
\[
{\rm \bf{i}}'=(\underbrace{1,\cdots,r}_{1{\rm \,st\
  cycle}},\underbrace{1,\cdots,r}_{2{\rm \,nd\
  cycle}},\cdots,\underbrace{1,\cdots,r}_{m-1{\rm \,th\ cycle}},
\underbrace{1,\cdots,i_n}_{m{\rm \,th\ cycle}},i_{n+1}),
\]
and set $\textbf{Y}'\in(\mathbb{C}^{\times})^{n+1}$ as
\[ \textbf{Y}':=(Y_{1,1},\cdots,Y_{1,r},\cdots,Y_{m-1,1},\cdots,Y_{m-1,r},Y_{m,1},\cdots,Y_{m,i_n},Y).
\] 
For an integer $k$ $(1\leq k\leq n)$, if $d:=i_k\neq i_{n+1}$, then $\Delta^L(k;{\rm \bf{i}}')(\textbf{Y}')$ does not depend on $Y$. So, we can regard it as a function on $(\mathbb{C}^{\times})^{n}$. Furthermore, we have 
\[
\Delta^L(k;{\rm \bf{i}})(\textbf{Y})=\Delta^L(k;{\rm \bf{i}}')(\textbf{Y}') .
\]
\end{lem}

By this lemma, when we calculate $\Delta^L(k;{\rm \bf{i}})(\textbf{Y})$, we may
assume that $i_n=i_k$ 
without loss of generality.

Let $B(\Lm_d)$ $(1\leq d\leq r)$ be the crystal base of the same type as $G$ and $v_{\Lm_d}\in B(\Lm_d)$ be its lowest weight vector. The following theorem is our main result.

\begin{thm}\label{thm1}
Let ${\rm \bf{i}}$ be the reduced word of $u\in W$ in $(\ref{redwords2})$ and $\textbf{Y}\in (\mathbb{C}^{\times})^n$ be the variables as in $(\ref{yset})$. We suppose that the 
index $i_k$ belongs to $(m-1)$th cycle. Setting $d:=i_k=i_n$, there exist positive integers $\{a_b|\ b\in B^-(\Lm_d)_{u_{\leq k}}\}$ such that
\begin{equation}\label{GenDem}
\Delta^L(k;{\rm \bf{i}})(\textbf{Y})= \sum_{b\in B^-(\Lm_d)_{u_{\leq k}}}a_b\ \mu(b),
\end{equation}
where $B^-(\Lm_d)_{u_{\leq k}}$ is the lower Demazure crystal of $B(\Lm_d)$ and $\mu:B(\Lm_d)\rightarrow \cY$ is the monomial realization of the crystal base $B(\Lm_d)$ such that $\mu(v_{\Lm_d})=\frac{1}{Y_{m,d}}$.
\end{thm}

\begin{ex}\label{monorealex2}
Let us consider the case of type ${\rm C}_2$. Take $u=s_1s_2s_1s_2\in W$ and let ${\rm \bf{i}}=(1,2,1,2)$ be its reduced word. Following Subsection \ref{SectFundC} and \ref{bilingen}, we shall calculate $\Delta^L(2;{\rm \bf{i}})(\textbf{Y})$. First, it follows from $\al_i(t)=t^{-h_i}$ and $y_i(t)={\rm exp}(t f_i)=1+t f_i$ on $V(\Lm_2)$ that
\begin{eqnarray*}
& &x^L_{{\rm \bf{i}}}(\textbf{Y})(v_1\wedge v_2) \\
&=& y_1(Y_{1,1})\al^{\vee}_1(Y^{-1}_{1,1})y_2(Y_{1,2})\al^{\vee}_2(Y^{-1}_{1,2})
y_1(Y_{2,1})\al^{\vee}_1(Y^{-1}_{2,1})y_2(Y_{2,2})
\al^{\vee}_2(Y^{-1}_{2,2})(v_1\wedge v_2)\\
&=& y_1(Y_{1,1})\al^{\vee}_1(Y^{-1}_{1,1})y_2(Y_{1,2})\al^{\vee}_2(Y^{-1}_{1,2}) 
\left(\frac{1}{Y_{2,1}}v_1+v_2\right)\wedge \left(\frac{Y_{2,1}}{Y_{2,2}}v_2+\frac{1}{Y_{2,1}}v_{\ovl{2}}+v_{\ovl{1}}\right)\\
&=& \left(\frac{1}{Y_{1,1}Y_{2,1}}v_1+\left(\frac{1}{Y_{2,1}}+\frac{Y_{1,1}}{Y_{1,2}}\right)v_2+\frac{1}{Y_{1,1}}v_{\ovl{2}}+v_{\ovl{1}} \right)\wedge \\
& &\qq \qq \left(\frac{Y_{1,1}Y_{2,1}}{Y_{1,2}Y_{2,2}}v_2+\left(\frac{Y_{2,1}}{Y_{1,1}Y_{2,2}}+\frac{Y_{1,2}}{Y_{2,1}Y_{1,1}} \right)v_{\ovl{2}}+\left(\frac{Y_{2,1}}{Y_{2,2}}+\frac{Y_{1,2}}{Y_{2,1}}+Y_{1,1} \right)v_{\ovl{1}} \right).
\end{eqnarray*}
Since $u_{\leq2}=s_1s_2$, we obtain
\begin{eqnarray}
\Delta^L(2;{\rm \bf{i}})(\textbf{Y})
&=&\lan x^L_{{\rm \bf{i}}}(\textbf{Y})(v_1\wedge v_2) ,\ s_1s_2(v_1\wedge v_2)\ran= \lan x^L_{{\rm \bf{i}}}(\textbf{Y})(v_1\wedge v_2) ,\ v_2\wedge v_{\ovl{1}} \ran \nonumber \\
&=&\left(\frac{1}{Y_{2,1}}+\frac{Y_{1,1}}{Y_{1,2}}\right)\cdot\left(\frac{Y_{2,1}}{Y_{2,2}}+\frac{Y_{1,2}}{Y_{2,1}}+Y_{1,1} \right)-\frac{Y_{1,1}Y_{2,1}}{Y_{1,2}Y_{2,2}} \nonumber \\
&=&\frac{Y^2_{1,1}}{Y_{1,2}}+2\frac{Y_{1,1}}{Y_{2,1}} +\frac{Y_{1,2}}{Y^2_{2,1}}+\frac{1}{Y_{2,2}}.\label{mono2exp}
\end{eqnarray}
Note that these four terms coincide with the monomials in Example \ref{monorealex1} $(\ref{monorealdia1})$ except for $Y_{0,2}$. The set $\{\frac{Y^2_{1,1}}{Y_{1,2}},\ \frac{Y_{1,1}}{Y_{2,1}},\ \frac{Y_{1,2}}{Y^2_{2,1}},\ \frac{1}{Y_{2,2}}\}$ is a monomial realization of Demazure crystal $B^-(\Lm_2)_{s_1s_2}$.
\end{ex}

\begin{ex}\label{monorealspinex2}
Let us consider the case of type ${\rm B}_3$. Take the element $u=s_1s_2s_3s_1s_2s_3s_1s_2s_3\in W$ and let ${\rm \bf{i}}=(1,2,3,1,2,3,1,2,3)$ be its reduced word. Following Subsection \ref{SectFundB} and \ref{bilingen}, we shall calculate $\Delta^L(6;{\rm \bf{i}})(\textbf{Y})$:
\[
\Delta^L(6;{\rm \bf{i}})(\textbf{Y})=\frac{Y_{1,2}}{Y_{1,3}}+\frac{Y_{1,3}Y_{2,1}}{Y_{2,2}}+\frac{Y_{1,3}}{Y_{3,1}}
+\frac{Y_{2,1}}{Y_{2,3}}+\frac{Y_{2,2}}{Y_{2,3}Y_{3,1}}+\frac{Y_{2,3}}{Y_{3,2}}+\frac{1}{Y_{3,3}}.
\]
Note that these seven terms coincide with the monomials in Example \ref{monorealexspin1} $(\ref{monorealspindia1})$ except for $Y_{0,3}$. The set $\{\frac{Y_{1,2}}{Y_{1,3}},\ \frac{Y_{1,3}Y_{2,1}}{Y_{2,2}},\ \frac{Y_{1,3}}{Y_{3,1}},\ \frac{Y_{2,1}}{Y_{2,3}},\ \frac{Y_{2,2}}{Y_{2,3}Y_{3,1}},\ \frac{Y_{2,3}}{Y_{3,2}},\ \frac{1}{Y_{3,3}}\}$ is a monomial realization of Demazure crystal $B^-(\Lm_3)_{s_1s_2s_3s_1s_2s_3}$.
\end{ex}

\section{The explicit formula of $\Delta^L(k;{\rm \bf{i}})$ of type ${\rm D}_r$}

In \cite{KaN:2015, KaN2:2016, Ka1:2016}, we have given the explicit formula of minors $\Delta^L(k;{\rm \bf{i}})$ of type ${\rm A}_r$, ${\rm B}_r$ and ${\rm C}_r$. In this section, we shall give the one for type ${\rm D}_r$ in some special cases, which is needed in the proof of Theorem \ref{thm1}. Let ${\rm \bf{i}}$ be a reduced word in (\ref{redwords}) and $i_k$ be its $k$-th the index from the left. We suppose that $i_k$ belongs to $(m-1)$ th cycle as in the previous section.

\subsection{The explicit formula of $\Delta^L(k;{\rm \bf{i}})$}

For $l\in[1,m]$ and $k\in J_{{\rm D}}$, we define the Laurent monomials $D(l,k)$ by
\begin{equation}\label{ddbar}
D(l,k):=
\begin{cases}
\frac{Y_{l,k-1}}{Y_{l,k}} & {\rm if}\ 1\leq k\leq r-2, \\
\frac{Y_{l,r-2}}{Y_{l,r-1}Y_{l,r}} & {\rm if}\ k= r-1, \\
\frac{Y_{l,r-1}}{Y_{l+1,r}} & {\rm if}\ k=r, \\
\frac{Y_{l,r}Y_{l,r-1}}{Y_{l+1,r-2}} & {\rm if}\ k= \ovl{r-1}, \\
\frac{Y_{l,|k|}}{Y_{l+1,|k|-1}} & {\rm if}\ \ovl{r-2}\leq k\leq \ovl{1},\ {\rm or}\ k=\ovl{r}.
\end{cases}
\end{equation}
and set formally
\begin{equation}\label{ddbar2} D(l,r+1):=\frac{1}{Y_{l,r}}. \end{equation}

We shall prove the following theorems:

\begin{thm}\label{thmd1}
In the setting of Theorem \ref{thm1}, suppose that $i_k=d<r-1$. We have
\begin{equation}\label{D-explicit}
 \Delta^L(k;{\rm \bf{i}})(\textbf{Y})=\sum_{(*)} D(l_1,k_1)D(l_2,k_2)\cdots D(l_d,k_d),
\end{equation}
\[ l_i:=
\begin{cases}
m-k_i+i & {\rm if}\ 1\leq k_i\leq r-1, \\
m-r+i & {\rm if}\ k_i\in\{\ovl{r},\cdots,\ovl{1}\}\cup\{r\},
\end{cases}
\]
where $(*)$ is the condition for $k_i$ $(1\leq i\leq d)$ $:$
\begin{equation}\label{D-cond1}
 k_1\ngeq k_2\ngeq \cdots\ngeq k_d,
\end{equation}
\begin{equation}\label{D-cond2}
i\leq k_i\leq m-1+i\qq (1\leq i\leq r-m),
\end{equation}
\begin{equation}\label{D-cond3}
i\leq k_i\leq \ovl{d-i+1}\qq (r-m< i\leq d).
\end{equation}
\end{thm}

\begin{thm}\label{thmd2}
In the setting of Theorem \ref{thm1}, suppose that $i_k=r$ (resp. $i_k=r-1$). We have
\[
\Delta^L(k;{\rm \bf{i}})(\textbf{Y})=\sum_{(*)} D(m-1,\ovl{k_1})D(m-2,\ovl{k_2})\cdots D(m-s,\ovl{k_s})D(m-s,r+1),
\]
where $(*)$ is the condition for $s\in\mathbb{Z}_{\geq0}$ and $k_i$ $(1\leq i\leq s)$ $:$ 
\begin{equation}\label{D-spin-set}
0\leq s\leq m-1\ {\rm and}\ s\ {\rm is\ even\ (resp.\ odd),} \qq 1\leq k_1<k_2<\cdots<k_s\leq r.
\end{equation}
\end{thm}

\subsection{The set $X_d(m,m-1)$ of paths}

In this subsection, we shall introduce a set $X_d(m,m-1)$ of paths
which will correspond to the set of the monomials in $\Delta^L(k;{\rm \bf{i}})(\textbf{Y})$. Let
$m$ and $d$ be the positive integers as in 
Sect.\ref{gmc}. Let $J_{{\rm D}}:=\{j,\ovl{j}|\ 1\leq j\leq r\}$ be as in Subsection \ref{SectFundD} and for $l\in\{1,2,\cdots,r\}$, set $|l|=|\ovl{l}|=l$.
\begin{defn}
Let us define the directed graph $(V_d,E_d)$ as follows:
The set $V_d=V_d(m)$ of vertices is defined by 
\[V_d(m):=\{{\rm vt}(m-s;a^{(s)})|\ 0\leq s\leq m,\
 a^{(s)}=(a^{(s)}_1,a^{(s)}_2,\cdots,a^{(s)}_d)\in J^d_D \}. 
\]
And we define the set $E_d=E_d(m)$ of directed edges as 
\begin{multline*} E_d(m):=\{{\rm vt}(m-s;a^{(s)})\rightarrow
{\rm vt}(m-s-1;a^{(s+1)})\\
|\ 0\leq s\leq m-1,\ {\rm vt}(m-s;a^{(s)}),\ {\rm vt}(m-s-1;a^{(s+1)})\in V_d(m)\}.
\end{multline*}
\end{defn}
\begin{defn}\label{Dpathdef}
Let $X_d(m,m-1)$ be the set of directed paths $p$
\begin{multline*}p={\rm vt}(m;a^{(0)}_1,\cdots,a^{(0)}_d)
\rightarrow{\rm vt}(m-1;a^{(1)}_1,\cdots,a^{(1)}_d)\rightarrow {\rm vt}(m-2;a^{(2)}_1,
\cdots,a^{(2)}_d)\\
\rightarrow
\cdots\rightarrow{\rm vt}(1;a^{(m-1)}_1,\cdots,a^{(m-1)}_d)
\rightarrow{\rm vt}(0;a^{(m)}_1,\cdots,a^{(m)}_d),
\end{multline*}
which satisfy the following conditions (i)-(v) : For $s\in\mathbb{Z}$ $(0\leq s\leq m)$,

\begin{itemize}
\item[(i)] $a^{(s)}_{1}<a^{(s)}_{2}<\cdots<a^{(s)}_{d}$ in the order (\ref{D-order}), 
\item[(ii)] If $1\leq a^{(s)}_{\zeta}\leq r-2$, then $a^{(s+1)}_{\zeta}=a^{(s)}_{\zeta}$ or $a^{(s)}_{\zeta}+1$. If $r-1\leq a^{(s)}_{\zeta}\leq \ovl{1}$, then $a^{(s)}_{\zeta}\leq a^{(s+1)}_{\zeta}\leq\ovl{1}$,
\item[(iii)] $(a^{(0)}_1,a^{(0)}_2,\cdots,a^{(0)}_d)=(1,2,\cdots,d)$ and
\[ (a^{(m)}_1,\cdots,a^{(m)}_d)=
\begin{cases}
(m,m+1,\cdots,r-1,\ovl{d-r+m},\cdots,\ovl{2},\ovl{1}) & {\rm if}\ m+d>r,\\
(m,m+1,\cdots,m+d-1) & {\rm if}\ m+d\leq r,
\end{cases}
\]
\item[(iv)] If $a^{(s+1)}_{\zeta}\in\{r,\ovl{r},\ovl{r-1},\cdots,\ovl{1}\}$, then $a^{(s+1)}_{\zeta}\ngeq a^{(s)}_{\zeta+1}$ in the order (\ref{D-order}). 
\end{itemize}

We say that two vertices ${\rm vt}(m-s;a^{(s)}_1,\cdots,a^{(s)}_d)$ and ${\rm vt}(m-s-1;a^{(s+1)}_1,\cdots,a^{(s+1)}_d)$ in $V_d(m)$ are {\it connected} if the conditions (i), (ii) and (iv) in Definition \ref{Dpathdef} are satisfied. And for a path $p$ as above, we call the sequence $a^{(0)}_{i}\rightarrow a^{(1)}_{i}\rightarrow a^{(2)}_{i}\rightarrow \cdots\rightarrow a^{(m)}_{i}$ the $i$-sequence of $p$.
\end{defn}

We define Laurent monomials associated with edges in $E_d(m)$.
\begin{defn}\label{Dlabeldef}

\begin{itemize}
\item[(i)] For each $s$ $(0\leq s\leq m)$ and $i,j\in J_{{\rm D}}$ with $i\leq j$, we set
\[Q^{(s)}(i\rightarrow j):=
\begin{cases}
\frac{Y_{m-s,i-1}}{Y_{m-s,i}} & {\rm if}\ j=i,\ i\in[1,r-2], \\
1 & {\rm if}\ j>i,\ i\in[1,r-2],
\end{cases}
\]
\[
Q^{(s)}(r-1\rightarrow j):=
\begin{cases}
\frac{Y_{m-s,r-2}}{Y_{m-s,r-1}Y_{m-s,r}} & {\rm if}\ j=r-1,\\
\frac{1}{Y_{m-s,r}} & {\rm if}\ j=r, \\
\frac{1}{Y_{m-s,|j|-1}} & {\rm if}\ j\in\{\ovl{r},\ovl{r-1},\cdots,\ovl{1}\}, \\
\end{cases} 
\]
\[Q^{(s)}(r\rightarrow j):=
\begin{cases}
\frac{Y_{m-s,r-1}}{Y_{m-s,r}} & {\rm if}\ j=r, \\
\frac{Y_{m-s,r-1}}{Y_{m-s,|j|-1}} & {\rm if}\ j\in\{\ovl{r-1},\cdots,\ovl{1}\}, \\
\end{cases} 
\ Q^{(s)}(i\rightarrow j):=
\begin{cases}
\frac{Y_{m-s,r}}{Y_{m-s,|j|-1}} & {\rm if}\ i=\ovl{r}, \\
\frac{Y_{m-s,r}Y_{m-s,r-1}}{Y_{m-s,|j|-1}} & {\rm if}\ i=\ovl{r-1}, \\
\frac{Y_{m-s,|i|}}{Y_{m-s,|j|-1}} & {\rm if}\ i\in\{\ovl{r-2},\cdots,\ovl{2},\ovl{1}\}.
\end{cases}
\]
For an edge $e^{i_1,\cdots,i_d}_{j_1,\cdots,j_d}={\rm vt}(m-s;i_1,\cdots,i_d)\rightarrow {\rm vt}(m-s-1;j_1,\cdots,j_d)$ in $E_d(m)$ such that $i_k\leq j_k$ $(k=1,2,\cdots,d)$, we define the {\it label} $Q^{(s)}(e^{i_1,\cdots,i_d}_{j_1,\cdots,j_d})$ {\it of the edge} $e^{i_1,\cdots,i_d}_{j_1,\cdots,j_d}$ as
\[ Q^{(s)}(e^{i_1,\cdots,i_d}_{j_1,\cdots,j_d}):=\prod^d_{k=1} Q^{(s)}(i_k\rightarrow j_k). \]
\item[(ii)] Let $p\in X_d(m,m-1)$ be a path:
\begin{multline*}p={\rm vt}(m;a^{(0)}_1,\cdots,a^{(0)}_d)
\rightarrow{\rm vt}(m-1;a^{(1)}_1,\cdots,a^{(1)}_d)\rightarrow {\rm vt}(m-2;a^{(2)}_1,
\cdots,a^{(2)}_d)\\
\rightarrow
\cdots\rightarrow{\rm vt}(1;a^{(m-1)}_1,\cdots,a^{(m-1)}_d)
\rightarrow{\rm vt}(0;a^{(m)}_1,\cdots,a^{(m)}_d).
\end{multline*}
For each $s$ $(0\leq s\leq m-1)$, we denote the label of the $(m-s)$ th edge ${\rm vt}(m-s;a^{(s)}_1,\cdots,a^{(s)}_d)\rightarrow{\rm vt}(m-s-1;a^{(s+1)}_1,\cdots,a^{(s+1)}_d)$ of $p$ by $Q^{(s)}(p)$.
And we define the {\it label} $Q(p)$ {\it of the path} $p$ as the total product:
\[ Q(p):=\prod_{s=0}^{m-1}Q^{(s)}(p). 
\]
\item[(iii)] For a subpath $p'$
\[
p'={\rm vt}(m-s';a^{(s')})\rightarrow{\rm vt}(m-s'-1;a^{(s'+1)})\rightarrow \\
\cdots\rightarrow{\rm vt}(m-s'';a^{(s'')})
\]
of $p$ $(0\leq s'<s''\leq m)$, we define the {\it label} {\it of the subpath} $p'$ as 
\[
Q(p'):=\prod_{s=s'}^{s''-1}Q^{(s)}(p). 
\]
\end{itemize}

\end{defn}

\subsection{One-to-one correspondence between paths in $X_d(m,m-1)$ and monomials in $\Delta^L(k;{\rm \bf{i}})(\textbf{Y})$}

We define $x_{-i}(Y)=y_i(Y)\cdot Y^{-h_i}$ in $(\ref{alxmdef})$ $(Y\in \mathbb{C}^{\times})$. 
For $1\leq s\leq m$ and $i_1,\cdots,i_d\in J_D=\{i,\ovl i|1\leq i\leq r\}$, let us define
\begin{equation}\label{xmdef1}
x^{(s)}_{-[1,r]}:=x_{-1}(Y_{s,1})\cdots x_{-r}(Y_{s,r}).
\end{equation}
and 
\begin{equation}\label{xmdef2}
 (s;i_1,i_2,\cdots,i_d):=\lan x^{(1)}_{-[1,r]} x^{(2)}_{-[1,r]}\cdots x^{(s)}_{-[1,r]}
(v_{i_1}\wedge\cdots\wedge v_{i_d}), \ u_{\leq k}(v_1\wedge\cdots\wedge v_d)  \ran . 
\end{equation}
Since $i_k$ belongs to $(m-1)$ th cycle of {\rm \bf{i}}, by (\ref{smpl}) we have
\begin{equation}\label{xmdef3}
u_{\leq k}(v_1\wedge\cdots\wedge v_d)=
\begin{cases}
v_m\wedge\cdots\wedge v_{m+d-1} & {\rm if}\ d\leq r-m, \\
v_m\wedge\cdots\wedge v_{r-1}\wedge v_{\ovl{d-r+m}}\wedge\cdots\wedge v_{\ovl{1}} & {\rm if}\ d>r-m.
\end{cases}
\end{equation}
We can verify $\Delta^L(k;{\rm \bf{i}})(\textbf{Y})=(m;1,\cdots,d)$. Following Subsection \ref{SectFundD}, we get the following:

\begin{lem}\label{xmprod2}We obtain
\[
x^{(s)}_{-[1,r]} v_j= \frac{Y_{s,j-1}}{Y_{s,j}}v_j+v_{j+1},\q x^{(s)}_{-[1,r]} v_{\ovl{j}}=\sum^{j}_{i=1} \frac{Y_{s,j}}{Y_{s,i-1}}v_{\ovl{i}}\q (1\leq j\leq r-2)
\]
\[
x^{(s)}_{-[1,r]} v_{r-1}= \frac{Y_{s,r-2}}{Y_{s,r-1}Y_{s,r}}v_{r-1}+\frac{1}{Y_{s,r}}v_{r}+\sum^r_{i=1} \frac{1}{Y_{s,i-1}}v_{\ovl{i}},\ \ x^{(s)}_{-[1,r]} v_{r}=\frac{Y_{s,r-1}}{Y_{s,r}}v_{r}+\sum^{r-1}_{i=1} \frac{Y_{s,r-1}}{Y_{s,i-1}}v_{\ovl{i}}
\]
\[
x^{(s)}_{-[1,r]} v_{\ovl{r}}=\sum^{r}_{i=1} \frac{Y_{s,r}}{Y_{s,i-1}}v_{\ovl{i}},\q 
x^{(s)}_{-[1,r]} v_{\ovl{r-1}}=\sum^{r-1}_{i=1} \frac{Y_{s,r-1}Y_{s,r}}{Y_{s,i-1}}v_{\ovl{i}}.
\]
\end{lem}


\begin{lem}\label{drc}

\begin{itemize} Let $(s;i_1,i_2,\cdots,i_d)$ $(1\leq s\leq m)$ be as in $(\ref{xmdef2})$,
\item[(i)] if $i_1\ngeq \cdots \ngeq i_d$ and there is no number $j\in[1,d-1]$ such that $i_j=r$, $i_{j+1}=\ovl{r}$ or $i_j=\ovl{r}$, $i_{j+1}=r$ then
\begin{equation}\label{xmlem120}
(s;i_1,i_2,\cdots,i_d)=\sum_{V} \prod^d_{t=1}Q^{(m-s)}(i_t\rightarrow j_t)(s-1;j_1,\cdots,j_d),
\end{equation}
where $V:=\{(j_1,\cdots,j_d)|\ i_t\leq j_t,\ {\rm if}\ i_t\leq r-2\ {\rm then}\ j_t=i_t\ {\rm or}\ i_t+1,\ {\rm if}\ i_t\geq r-1\ {\rm then}\ j_t\ngeq i_{t+1} \}$.
\item[(ii)] We suppose that $s\leq m-r+j-1$ and $i_1<\cdots<i_{j-1}$ for some $1\leq j\leq d$.
If $i_{j-1}\leq r-1$, $\ovl{r-1}\leq i_{j+1}$ and either $i_{j}=r$ or $i_{j}=\ovl{r}$, then $(s;i_1,i_2,\cdots,i_d)=0$.
\item[(iii)] We suppose that $i_1\ngeq \cdots \ngeq i_d$ and $s\leq m-r+j-1$ for some $1\leq j\leq d$. If $i_{j-2}\leq r-1$, $\ovl{r-1}\leq i_{j+1}$ and either $i_{j-1}=r$, $i_{j}=\ovl{r}$ or $i_{j-1}=\ovl{r}$, $i_{j}=r$ then $(s;i_1,i_2,\cdots,i_d)=0$.
\end{itemize}

\end{lem}

\begin{proof}
(i) Let us calculate $x^{(s)}_{-[1,r]}(v_{i_1}\wedge\cdots\wedge v_{i_d})$. It follows from Lemma \ref{xmprod2} and Definition \ref{Dlabeldef} (i) that
\begin{equation}\label{xmlem12}
x^{(s)}_{-[1,r]}(v_{i_1}\wedge\cdots\wedge v_{i_d})=\sum_{V'} \prod^{d}_{t=1} Q^{(m-s)}(i_t\rightarrow j_t)v_{j_1}\wedge\cdots\wedge v_{j_d},  
\end{equation}
where $(j_1,\cdots,j_d)$ runs over the set $V':=\{(j_1,\cdots,j_d)|\ i_t\leq j_t,\ {\rm if}\ i_t\leq r-2\ {\rm then}\ j_t=i_t\ {\rm or}\ i_t+1\}$. Now we set $V'':=\{(j_1,\cdots,j_d)\in V'|\ \exists l\ {\rm such\ that}\ |j_{l}|\leq |i_{l+1}|\}$. We define the map $\tau:V''\rightarrow V''$ as follows: Take $(j_1,\cdots,j_d)\in V''$. Let $l$ $(1\leq l\leq d-1)$ be the index such that $|j_{1}|>|i_{2}|,\ |j_{2}|>|i_{3}|,\cdots, |j_{l-1}|>|i_{l}|$ and $|j_{l}|\leq |i_{l+1}|$. Since $|j_{l+1}|\leq |i_{l+1}|$ by the definition of $V'$, we have $(j_1,\cdots,j_{l+1},j_{l},\cdots,j_d)\in V''$. So, we define $\tau(\cdots,j_{l},j_{l+1},\cdots):= (\cdots,j_{l+1},j_{l},\cdots)$. We can easily see that $\tau^2=id_{V''}$.

In $(\ref{xmlem12})$, $(v_{j_1}\wedge \cdots\wedge j_l\wedge j_{l+1},\cdots,j_d)$ and $(v_{j_1}\wedge\cdots\wedge j_{l+1}\wedge j_l\wedge\cdots\wedge j_d)$ have the same coefficient and then they are cancelled, which yields that
\[x^{(s)}_{-[1,r]}(v_{i_1}\wedge\cdots\wedge v_{i_d})=\sum_{V} \prod^{d}_{t=1} Q^{(m-s)}(i_t\rightarrow j_t)v_{j_1}\wedge\cdots\wedge v_{j_d}, \]
since $V=V'\setminus V''$. The definition (\ref{xmdef2}) implies our desired result. 

(ii) We use induction on $s$. We suppose that $i_{j-1}\leq r-1$, $i_{j}=r$. By the same argument as in (i), $(s;i_1,i_2,\cdots,i_d)$ is a linear combination of $\{(s-1;l_1,\cdots,l_d)\}$ such that $l_{j-2}\leq r-1$ and $l_{j-1}\leq r-1$ or $l_{j-1}=\ovl{r}$. In the case $l_{j-1}=\ovl{r}$, $(s-1;l_1,\cdots,l_d)=0$ by the assumption of induction. Hence, we may assume that $l_{j-1}\leq r-1$. Then $(s-1;l_1,\cdots,l_d)$ is a linear combination of $(0;\zeta_1,\cdots,\zeta_d)$ such that 
$\#\{1\leq \al\leq d|\ \ovl{r}\leq\zeta_{\al}\leq\ovl{1}\}\leq(s-2)+(d-j+2)=s+d-j\leq m-r+d-1$, which implies that $(0;\zeta_1,\cdots,\zeta_d)=0$ by (\ref{xmdef2}) and (\ref{xmdef3}). Hence we obtain $(s;i_1,i_2,\cdots,i_d)=0$. Similarly, we can verify that if $i_{j-1}\leq r-1$, $i_{j}=\ovl{r}$ then $(s;i_1,i_2,\cdots,i_d)=0$.

\nd
(iii) Similar to (ii), we use induction on $s$. We suppose that $i_{j-1}=r$, $i_{j}=\ovl{r}$.
By the same argument as in (i), $(s;i_1,i_2,\cdots,i_d)$ is a linear combination of $\{(s-1;l_1,\cdots,l_d)\}$ such that $i_{t+1}\nleq l_{t}$ and either $l_{j-1}=r$ or $l_{j}=\ovl{r}$. If $l_{j-1}=r$ then we may assume that $l_{j-2}\leq r-1$ by the assumption of induction. In this case, however, $(s-1;l_1,\cdots,l_d)=0$ by (ii), which implies that $(s;i_1,i_2,\cdots,i_d)=0$. If $l_{j}=\ovl{r}$ then we also get $l_{j-2}\leq r-1$, and $(s-1;l_1,\cdots,l_d)$ is a linear combination of $\{(s-2;m_1,\cdots,m_d)\}$ such that $m_{j-2}=r-1$ or $r$. Using (ii), we have $(s-2;m_1,\cdots,m_d)=0$, which implies $(s;i_1,i_2,\cdots,i_d)=0$. 

\end{proof}

Supposing $s\leq m-r+j-1$ and writing $(m;1,\cdots,d)$ as a linear combination of $\{(s;i_1,\cdots,i_d)|\ 1\leq i_1,\cdots,i_d\leq\ovl{1}\}$, then the coefficient of $(s;i_1,\cdots,i_d)$ such that either $i_1<\cdots<i_{j-1}=r$, $i_{j}=\ovl{r}$ or $i_1<\cdots<i_{j-1}=\ovl{r}$, $i_{j}=r$ is $0$. 

\begin{prop}\label{pathlemD} We have
\[ \Delta^L(k;{\rm \bf{i}})(\textbf{Y})=\sum_{p\in X_d(m,m-1)} Q(p). \]
\end{prop}

\begin{proof}

We suppose that $1\leq i_1<\cdots<i_d\leq\ovl{1}$ and $1\leq s\leq m$. By the definition of $V$ in Lemma \ref{drc}, we see that $(j_1,\cdots,j_d)\in V$ if and only if the vertices vt$(s-1;j_1,\cdots,j_d)$ and vt$(s;i_1,\cdots,i_d)$ are connected (Definition \ref{Dpathdef}). Further, the coefficient of $(s-1;j_1,\cdots,j_d)$ in $(\ref{xmlem120})$ coincides with the label of the edge $e^{i_1,\cdots,i_d}_{j_1,\cdots,j_d}$ between vt$(s;i_1,\cdots,i_d)$ and vt$(s-1;j_1,\cdots,j_d)$ (Definition \ref{Dlabeldef} (i)). Hence, we get
\begin{equation}\label{plpr1}
 (s;i_1,\cdots,i_d)=\sum_{(j_1,\cdots,j_d)} Q^{(m-s)}(e^{i_1,\cdots,i_d}_{j_1,\cdots,j_d}) \cdot(s-1;j_1,\cdots,j_d),
\end{equation}
where $(j_1,\cdots,j_d)$ runs over the set $\{(j_1,\cdots,j_d)|\ {\rm vt}(s-1;j_1,\cdots,j_d)$ and ${\rm vt}(s;i_1$,
$\cdots,i_d)\ {\rm are\ connected}\}$. Since we know that if $j_l\in\{\ovl{r},\cdots,\ovl{1}\}\cup\{0\}$ then $j_l\leq i_{l+1}$, if $i_l\leq r-1$ then $j_l=i_l$ or $i_l+1$, and $i_{l+1}\leq j_{l+1}$ in $V$, one gets $j_l<j_{l+1}$, and then $j_1<j_2<\cdots<j_d$. The followings is obtained in the same way as (\ref{plpr1}):
\begin{equation}\label{plpr2}
 (s-1;j_1,\cdots,j_d)=\sum_{(k_1,\cdots,k_d)} Q^{(m-s+1)}(e^{j_1,\cdots,j_d}_{k_1,\cdots,k_d}) \cdot(s-2;k_1,\cdots,k_d),
\end{equation}
where $(k_1,\cdots,k_d)$ runs over the set $\{(k_1,\cdots,k_d)|\ {\rm vt}(s-2;k_1,\cdots,k_d)$ and ${\rm vt}(s-1;j_1,\cdots,j_d)\ {\rm are\ connected}\}$ and $e^{j_1,\cdots,j_d}_{k_1,\cdots,k_d}$ is the edge between vertices ${\rm vt}(s-1;j_1,\cdots,j_d)$ and ${\rm vt}(s-2;k_1,\cdots,k_d)$. By (\ref{plpr1}), (\ref{plpr2}), $(s;i_1,\cdots,i_d)$ is a linear combination of $\{(s-2;k_1,\cdots,k_d)\}$, and the coefficient of $(s-2;k_1,\cdots,k_d)$ is as follows:
\[ \sum_{(j_1,\cdots,j_d)} Q^{(m-s)}(e^{i_1,\cdots,i_d}_{j_1,\cdots,j_d}) \cdot Q^{(m-s+1)}(e^{j_1,\cdots,j_d}_{k_1,\cdots,k_d}) \cdot(s-2;k_1,\cdots,k_d),\]
where $(j_1,\cdots,j_d)$ runs over the set $\{(j_1,\cdots,j_d)|\ {\rm vt}(s-1;j_1,\cdots,j_d)$ is connected to the vertices vt$(s;i_1,\cdots,i_d)$ and vt$(s-2;k_1,\cdots,k_d)\}$. The coefficient of $(s-2;k_1,\cdots,k_d)$ coincides with the label of subpath (Definition \ref{Dlabeldef} (iii))
\[{\rm vt}(s;i_1,\cdots,i_d)\rightarrow {\rm vt}(s-1;j_1,\cdots,j_d)\rightarrow{\rm vt}(s-2;k_1,\cdots,k_d).\]
Repeating this argument, we see that $(s;i_1,\cdots,i_d)$ is a linear combination of $\{(0;l_1,\cdots,l_d)\}$ $(1\leq l_1<\cdots<l_d\leq \ovl{1})$. The coefficient of $(0;l_1,\cdots,l_d)$ is equal to the sum of labels of all subpaths from vt$(s;i_1,\cdots,i_d)$ to vt$(0;l_1,\cdots,l_d)$. In the case $m'+d>r$ (resp. $m'+d\leq r$), for $1\leq l_1<\cdots<l_d\leq \ovl{1}$, if $(l_1,\cdots,l_d)=(m'+1,m'+2,\cdots,r,\ovl{d-r+m'},\cdots,\ovl{2},\ovl{1})$ (resp. $=(m'+1,m'+2,\cdots,m'+d)$), then we obtain $(0;l_1,\cdots,l_d)=1$ by (\ref{xmdef2}). If $(l_1,\cdots,l_d)$ is not as above, we obtain $(0;l_1,\cdots,l_d)=0$. Therefore, we see that $(s;i_1,\cdots,i_d)$ is equal to the sum of labels of subpaths from vt$(s;i_1,\cdots,i_d)$ to vt$(0;m,m+1,\cdots,r-1,\ovl{d-r+m},\cdots,\ovl{2},\ovl{1})$ (resp. vt$(0;m,m+1,\cdots,m+d-1))$.
In particular, $\Delta^L(k;{\rm \bf{i}})(\textbf{Y})=(m;1,2,\cdots,d)$ is equal to the sum of labels of paths in $X_d(m,m-1)$, which shows $\Delta^L(k;{\rm \bf{i}})(\textbf{Y})=\sum_{p\in
X_d(m,m-1)} Q(p)$. 
\end{proof}

\subsection{The properties of paths in $X_d(m,m-1)$}

\begin{lem}\label{fixedeqlem}
Let $p\in X_d(m,m-1)$ be a path in the form
\begin{multline}\label{fixpath}
p={\rm vt}(m;a^{(0)}_1,\cdots,a^{(0)}_d)\rightarrow
\cdots\rightarrow{\rm vt}(2;a^{(m-2)}_1,\cdots,a^{(m-2)}_d)  \\ 
\rightarrow{\rm vt}(1;a^{(m-1)}_1,\cdots,a^{(m-1)}_d)\rightarrow{\rm vt}(0;a^{(m)}_1,\cdots,a^{(m)}_d).
\end{multline}

\begin{itemize} 
\item[(i)] For $i$ with $r-m+1\leq i\leq d$, we obtain
\begin{equation}\label{fixedeq}
a^{(m)}_i=a^{(m-1)}_i=\cdots=a^{(r-i+1)}_i=\ovl{d-i+1}.
\end{equation}
\item[(ii)]
Let $a^{(0)}_{i}\rightarrow a^{(1)}_{i}\rightarrow a^{(2)}_{i}\rightarrow \cdots\rightarrow a^{(m)}_{i}$ be the $i$-sequence of the path $p$ $(Definition\ \ref{Dpathdef})$.

\begin{itemize} 
\item[(1)] In the case $i\leq r-m$, one has
\[   \#\{0\leq s\leq m-1|\ 1\leq a^{(s)}_i\leq r-1,\ {\rm and}\ a^{(s)}_i=a^{(s+1)}_i\}=1. \]
\item[(2)] In the case $i> r-m$, one has
\[
\#\{0\leq s\leq r-i+1|\ 1\leq a^{(s)}_i\leq r-1\, {\rm and}\ a^{(s)}_i=a^{(s+1)}_i\}+
\#\{0\leq s\leq r-i+1|\ a^{(s)}_i\in\{r,\ovl{r},\cdots,\ovl{1}\}\}=1. 
\]
\end{itemize}

\end{itemize}

\end{lem}

\begin{proof}
(i) By Definition \ref{Dpathdef} (iii), we get $a^{(m)}_{r-m+1}=\ovl{d-r+m}$. Using Definition \ref{Dpathdef} (iv) repeatedly, we obtain $\ovl{d-r+m}=a^{(m)}_{r-m+1}<a^{(m-1)}_{r-m+2}<a^{(m-2)}_{r-m+3}<\cdots<a^{(r-d+1)}_{d}\leq \ovl{1}$, which means $a^{(r-i+1)}_{i}=\ovl{d-i+1}$ $(r-m+1\leq i\leq d)$. It follows from Definition \ref{Dpathdef} (ii) and (iii) that $\ovl{d-i+1}=a^{(r-i+1)}_{i}\leq a^{(r-i+2)}_{i}\leq\cdots\leq a^{(m-1)}_{i}\leq a^{(m)}_{i}=\ovl{d-i+1}$, which yields (\ref{fixedeq}).

(ii) (1) In the case $i\leq r-m$, Definition \ref{Dpathdef} (ii) and (iii) show that 
\begin{equation}\label{tab-length1}
 i=a^{(0)}_i\leq a^{(1)}_i\leq\cdots\leq a^{(m)}_i=m+i-1,\ \ a^{(s+1)}_i=a^{(s)}_i\ {\rm or}\ a^{(s)}_i+1. 
\end{equation}
In particular, we get $1\leq a^{(s)}_i\leq r$ for $1\leq s\leq m$.
By (\ref{tab-length1}), we obtain
\[ \#\{0\leq s\leq m-1|\ a^{(s+1)}_i=a^{(s)}_i+1\}=m-1,
\]
which implies $\#\{0\leq s\leq m-1|\ a^{(s)}_i=a^{(s+1)}_i\}=1$. 

(2) In the case $i>r-m$, by Definition \ref{Dpathdef} (ii), we have
\[
 i=a^{(0)}_i\leq a^{(1)}_i\leq\cdots\leq a^{(r-i+1)}_i\leq \ovl{1}. 
\]
We suppose that 
\begin{equation}\label{tab-length2}
i= a^{(0)}_i\leq a^{(1)}_i\leq \cdots\leq a^{(l)}_i\leq r-1,\ {\rm and}\ \ r-1< a^{(l+1)}_i\leq \cdots\leq a^{(r-i)}_i\leq\ovl{1},
\end{equation}
for some $l$ $(1\leq l\leq r-i)$. Definition \ref{Dpathdef} (ii) implies that $a^{(s+1)}_i=a^{(s)}_i$ or $a^{(s)}_i+1$ $(0\leq s\leq l-1)$ and $a^{(l)}_i=r-1$. Therefore, 
\[ i=a^{(0)}_i\leq a^{(1)}_i\leq\cdots\leq a^{(l)}_i=r-1,\qq a^{(s+1)}_i=a^{(s)}_i\ {\rm or}\ a^{(s)}_i+1. \]
So, we have $\#\{0\leq s\leq l-1 |\ a^{(s+1)}_i=a^{(s)}_i \}=l-(r-1-i)=l-r+i+1$ in the same way as (2)(i).

On the other hand, the assumption $r-1< a^{(l+1)}_i\leq\cdots\leq a^{(r-i)}_i\leq \ovl{1}$ in (\ref{tab-length2}) means that $\{l+1\leq s\leq r-i |\ a^{(s)}_i\in \{r,\ovl{r},\cdots,\ovl{1}\}\}=\{l+1,l+2,\cdots, r-i\}$. Hence, $\#\{0\leq s\leq l-1 |\ a^{(s+1)}_i=a^{(s)}_i \}+\#\{l+1\leq s\leq r-i |\ a^{(s)}_i\in \{r,\ovl{r},\cdots,\ovl{1}\}\}=(l-r+i+1)+(r-i-l)=1$. 
\end{proof}

By this lemma, we define $l_i\in\{0,1,\cdots,m\}$ $(1\leq i\leq d)$ for the path $p\in X_d(m,m-1)$ in (\ref{fixpath}) as follows: For $i\leq r-m$, we define $0\leq l_i\leq m$ as the unique number which satisfies $a^{(l_i)}_i=a^{(l_i+1)}_i\leq r-1$ (Lemma \ref{fixedeqlem}(ii) (1)). For $i>r-m$, we define $l_i$ $(0\leq l_i\leq r-i+1)$ as the unique number which satisfies either $a^{(l_i)}_i=a^{(l_i+1)}_i\leq r-1$ or $a^{(l_i)}_i\in\{r,\ovl{r},\cdots,\ovl{1}\}$ (Lemma \ref{fixedeqlem}(ii) (2)). We also set $k_i\in \{j,\ovl{j}|\ 1\leq j\leq r\}$ $(1\leq i\leq d)$ as $k_i:=a^{(l_i)}_i$.

\begin{lem}\label{kpro2d} For a path $p$ as in the paragraph above, we have the following.

\begin{itemize} 
\item[(i)] For $1\leq i\leq d$,
\[ l_{i}=
\begin{cases}
k_{i}-i & {\rm if}\  k_{i}\in\{1,2,\cdots,r-1\}, \\
r-i & {\rm if}\ k_{i}\in\{r,\ovl{r},\ovl{r-1},\cdots,\ovl{1}\}.
\end{cases}
\]
\item[(ii)] For $1\leq i\leq d-1$, if $k_{i}\in\{1,2,\cdots,r-1\}$, then
\[ k_{i}<k_{i+1},\q l_i\leq l_{i+1}. \]
For $1\leq i\leq d-1$, if $k_{i}\in\{r,\ovl{r},\ovl{r-1},\cdots,\ovl{1}\}$, then
\[ k_{i}\ngeq k_{i+1},\q l_{i}=l_{i+1}+1. \]
\item[(iii)] One has $Q(p)=D(m-l_1,k_1)D(m-l_2,k_2)\cdots D(m-l_d,k_d)$.
\end{itemize}

\end{lem}

\begin{proof}
(i) First, we suppose that $k_{i}\in\{1,2,\cdots,r-1\}$. The definition of $l_i$ says
\[ a^{(0)}_i=i,\ a^{(1)}_i=i+1,\ a^{(2)}_i=i+2,\cdots,\ a^{(l_i)}_i=l_i+i. \]
Hence, $k_{i}=a^{(l_i)}_i=l_i+i$, which means that $l_i=k_i-i$. Next, we suppose that $k_{i}\in\{r,\ovl{r},\ovl{r-1},\cdots,\ovl{1}\}$. It follows from Lemma \ref{fixedeqlem} (i) that $l_i=r-i$.

(ii) Since $a^{(s)}_i<a^{(s)}_{i+1}$ for $0\leq s\leq m$, if $k_i\in\{1,2,\cdots,r-1\}$ then $k_{i}<k_{i+1}$. The inequality $l_i\leq l_{i+1}$ is obtained by (i). If $k_{i}\in\{r,\ovl{r},\ovl{r-1},\cdots,\ovl{1}\}$ then $k_{i}\ngeq k_{i+1}$ by Definition \ref{Dpathdef} (iv). It is easily to see $l_{i}=l_{i+1}+1$ by using (i).

(iii) By Definition \ref{Dlabeldef} (i) and Lemma \ref{fixedeqlem} (ii), one gets
\begin{multline*} 
Q^{(0)}(a^{(0)}_i\rightarrow a^{(1)}_i) Q^{(1)}(a^{(1)}_i\rightarrow a^{(2)}_i)
\cdots Q^{(m-1)}(a^{(m-1)}_i\rightarrow a^{(m)}_i) \\
=\begin{cases}
D(m-l_i,\ k_i) & {\rm if}\ i\leq r-m, \\
D(m-l_i,\ k_i)\cdot \frac{1}{Y_{m-r+i,d-i}}\prod^{m}_{s=r-i+1} \frac{Y_{m-s,d-i+1}}{Y_{m-s,d-i}} & {\rm if}\ i>r-m. 
\end{cases}
\end{multline*}
Therefore, we get
\[ Q(p)=\prod^{d}_{i=1}\prod^{m-1}_{s=0}Q^{(s)}(a^{(s)}_i\rightarrow a^{(s+1)}_i)
	=D(m-l_1,k_1)\cdots D(m-l_d,k_d).\]
\end{proof}

\subsection{The proof of Theorem \ref{thmd1}}

 We shall prove that
\begin{equation}\label{thmd1ins}
\Delta^L(k;{\rm \bf{i}})(\textbf{Y})=\sum_{(*)} D(m-l_1,k_1)D(m-l_2,k_2)\cdots D(m-l_d,k_d),
\end{equation}
\[ l_i:=
\begin{cases}
k_i-i & {\rm if}\ 1\leq k_i\leq r-1, \\
r-i & {\rm if}\ k_i\in\{\ovl{r},\cdots,\ovl{1}\}\cup\{r\},
\end{cases}
\]
which is equivalent to Theorem \ref{thmd1}.

[{\sl Proof of Theorem \ref{thmd1}.}]

As we have seen in Lemma \ref{kpro2d}, the monomial $Q(p)$ $(p\in X_d(m,m-1))$ is described as
\[ Q(p)=D(m-l_1,k_1)D(m-l_2,k_2)\cdots D(m-l_d,k_d) \]
with $k_1\ngeq k_2\ngeq \cdots \ngeq k_d$. Definition \ref{Dpathdef} (iii) and the definition of $l_i$ imply that $a^{(0)}_i=i\leq a^{(l_i)}_i=k_i \leq m+i-1=a^{(m)}_i$ $(i\leq r-m)$ and $a^{(0)}_i=i\leq a^{(l_i)}_i=k_i \leq \ovl{d-i+1}=a^{(m)}_i$ $(i>r-m)$. Therefore, these $\{k_i\}_{1\leq i\leq d}$ satisfy the conditions (\ref{D-cond1}), (\ref{D-cond2}) and (\ref{D-cond3}) in Theorem \ref{thmd1}.

Conversely, let $\{K_i\}_{1\leq i\leq d}$ be the sequence in $J_{{\rm D}}$ which satisfies the conditions (\ref{D-cond1}), (\ref{D-cond2}) and (\ref{D-cond3}) in Theorem \ref{thmd1}:
\begin{equation}\label{D-cond1dd}
 K_1\ngeq K_2\ngeq \cdots\ngeq K_d,
\end{equation}
\begin{equation}\label{D-cond2dd}
i\leq K_i\leq m-1+i\qq (1\leq i\leq r-m),
\end{equation}
\begin{equation}\label{D-cond3dd}
i\leq K_i\leq \ovl{d-i+1}\qq (r-m< i\leq d).
\end{equation}
By Proposition \ref{pathlemD}, it suffices to show that there exists a path $p\in X_d(m,m-1)$ such that
\begin{equation}\label{finclaimd}
Q(p)=D(m-L_1,K_1)D(m-L_2,K_2)\cdots D(m-L_d,K_d)
\end{equation}
and
\[ L_{i}:=
\begin{cases}
K_{i}-i & {\rm if}\  K_{i}\in\{1,\cdots,r-1\}, \\
r-i & {\rm if}\ K_{i}\in\{r,\ovl{r},\cdots,\ovl{1}\},
\end{cases}
\]
for $1\leq i\leq d$. Since we supposed $K_i\ngeq K_{i+1}$, we can easily verify if $K_i\in\{1,\cdots,r-1\}$ then $L_i\leq L_{i+1}$ and if $K_i\in\{r,\ovl{r},\cdots,\ovl{1}\}$ then $L_i=L_{i+1}+1$. We define a path $p={\rm vt}(m;a^{(0)}_1,\cdots,a^{(0)}_d)\rightarrow\cdots \rightarrow {\rm vt}(0;a^{(m)}_1,\cdots,a^{(m)}_d)\in X_d(m,m-1)$ as follows: For $i$ $(1\leq i\leq r-m)$, we define the $i$-sequence of $p$ as
\[ a^{(0)}_i=i,\ a^{(1)}_i=i+1, \cdots,\ a^{(L_i)}_i=L_i+i,\ a^{(L_i+1)}_i=L_i+i,\ a^{(L_i+2)}_i=L_i+i+1,\]
\[ a^{(L_i+3)}_i=L_i+i+2,\cdots, a^{(m)}_i=m+i-1. \]
For $i$ $(r-m<i\leq d)$, if $K_{i}\in\{1,2,\cdots,r-1\}$, we define
\[ a^{(0)}_i=i,\ a^{(1)}_i=i+1, \cdots,\ a^{(L_i)}_i=L_i+i,\ a^{(L_i+1)}_i=L_i+i,\ a^{(L_i+2)}_i=L_i+i+1,\]
\[ a^{(L_i+3)}_i=L_i+i+2,\cdots, a^{(r-i)}_i=r-1, a^{(r-i+1)}_i=a^{(r-i+2)}_i=\cdots=a^{(m)}_i=\ovl{d-i+1}, \]
if $K_{i}\in \{r,\ovl{r},\cdots,\ovl{1}\}$, define
\[ a^{(0)}_i=i,\ a^{(1)}_i=i+1, \cdots,\ a^{(r-i-1)}_i=r-1,\ a^{(r-i)}_i=K_{i},\ a^{(r-i+1)}_i=a^{(r-i+2)}_i=\cdots=a^{(m)}_i=\ovl{d-i+1}. \]
These mean that $a^{(L_i)}_i=K_i$. We can verify (\ref{finclaimd}) and the path $p$ satisfies the conditions of $X_d(m,m-1)$ in Definition \ref{Dpathdef}. Hence, by Proposition \ref{pathlemD}, we obtain (\ref{thmd1ins}). \qed

\subsection{The proof of Theorem \ref{thmd2}}

\begin{defn}\label{vspesp}
Let us define the directed graph $(V^{sp},E^{sp})$ as follows:
We define the set $V^{sp}=V^{sp}(m)$ of vertices as 
\[
V^{sp}(m):=\{{\rm vt}(s;k^{(s)}_1,\cdots,k^{(s)}_t)|0\leq t\leq s\leq m,\ 1\leq k^{(s)}_1<k^{(s)}_2<\cdots<k^{(s)}_t\leq r \}. 
\]
And we define the set $E^{sp}=E^{sp}(m)$ of directed edges as 
\begin{multline*} E^{sp}(m):=\{{\rm vt}(s;k^{(s)}_1,\cdots,k^{(s)}_t)\rightarrow
{\rm vt}(s-1;k^{(s+1)}_1,\cdots,k^{(s+1)}_{t'})\\
|\ 0\leq s\leq m-1,\ {\rm vt}(s;k^{(s)}_1,\cdots,k^{(s)}_t),\ {\rm vt}(s-1;k^{(s-1)}_1,\cdots,k^{(s-1)}_{t'})\in V^{sp}(m)\}.
\end{multline*}
\end{defn}

\begin{defn}\label{Dpathdef-sp}
Let $X^r(m,m-1)$  (resp. $X^{r-1}(m,m-1)$) be the set of directed paths $p$ in $(V^{sp},E^{sp})$
which satisfy the following conditions: For $s\in\mathbb{Z}$ $(1\leq s\leq m)$, if ${\rm vt}(s;k^{(s)}_1,\cdots,k^{(s)}_t)\rightarrow{\rm vt}(s-1;k^{(s-1)}_1,\cdots,k^{(s-1)}_{t'})$ is an edge included in $p$, then

\begin{itemize} 
\item[(i)] $t'=t$ or $t'=t+2$,
\item[(ii)] if $t'=t+2$, then $k^{(s-1)}_{t+2}=r$,
\item[(iii)] $k^{(s-1)}_i\leq k^{(s)}_i<k^{(s-1)}_{i+1}$,
\item[(iv)] the starting vertex of $p$ is ${\rm vt}(m;\phi)$ (resp. ${\rm vt}(m;r)$),
\item[(v)] if $m$ is odd (resp. even) then the ending vertex of $p$ is ${\rm vt}(0;1,2,\cdots,m-1)$, if $m$ is even (resp. odd) then one of $p$ is ${\rm vt}(0;1,2,\cdots,m-1,r)$.
\end{itemize}

\end{defn}

Define a Laurent monomial associated with each edge of paths in $X^r(m,m-1)$ and $X^{r-1}(m,m-1)$.
\begin{defn}\label{Dlabeldef-sp}
Let $p$ be a path in $X^r(m,m-1)$ or $X^{r-1}(m,m-1)$.
\item For each $s$ $(1\leq s\leq m)$, we define the {\it label of an edge} ${\rm vt}(s;i_1,i_2,\cdots,i_t)\rightarrow{\rm vt}(s-1;j_1,j_2,\cdots,j_{t'})$ in $p$ as the Laurent monomial which is given as follows and denote it $Q^{(s)}(p)$ or $Q({\rm vt}(s;i_1,i_2,\cdots,i_t)\rightarrow{\rm vt}(s-1;j_1,j_2,\cdots,j_{t'}))$: 
\[ Q^{(s)}(i_p\rightarrow j_p):=
\begin{cases}
\frac{Y_{m-s,i_p}}{Y_{m-s,j_p-1}} & {\rm if}\ i_p\leq r-2, \\
\frac{Y_{m-s,r-1}Y_{m-s,r}}{Y_{m-s,j_p-1}} & {\rm if}\ i_p=r-1, \\
\frac{Y_{m-s,r}}{Y_{m-s,j_p-1}} & {\rm if}\ i_p=r, \\
\end{cases} 
\ \ Q^{(s)}(p):=
\begin{cases}
\frac{1}{Y_{m-s,r}}\prod^t_{p=1}Q^{(s)}(i_p\rightarrow j_p) & {\rm if}\ t'=t, \\
\frac{1}{Y_{m-s,j_{t+1}-1}}\prod^t_{p=1}Q^{(s)}(i_p\rightarrow j_p) & {\rm if}\ t'=t+2. 
\end{cases}
.\]
And we define the {\it label} $Q(p)$ {\it of the path} $p$ as the total product:
\[Q(p):=\prod_{s=0}^{m-1}Q^{(s)}(p). 
\]
\end{defn}

\begin{prop}
In the setting of Theorem \ref{thmd2}, we suppose that $i_k=r-1$ or $r$.
\begin{equation}\label{Dpathlem-spin}
 \Delta^L(k;{\rm \bf{i}})(\textbf{Y})=\sum_{p\in X^{i_k}(m,m-1)} Q(p). 
\end{equation}
\end{prop}

\begin{proof}
We shall see the case $i_k=r$. Using the bilinear form (\ref{minor-bilin}), we obtain
\[ \Delta^L(k;{\rm \bf{i}})(\textbf{Y})=
 \Del_{u_{\leq k}\Lm_r,\Lm_r}(x^L_{{\rm \bf{i}}}(\textbf{Y}))
=\lan x^L_{{\rm \bf{i}}}(\textbf{Y})\cdot u_{\Lm_r}\, ,\, \ovl{u_{\leq k}}\cdot u_{\Lm_r}\ran,
\]
where $u_{\Lm_r}$ is the highest vector in the spin representation $V(\Lm_r)$. As we have seen in Subsection \ref{SectFundD}, we can describe basis vectors of $V(\Lm_r)$ as $(\ep_1,\cdots,\ep_r)$ $(\ep_i\in\{+,-\})$. In particular, the highest weight vector $u_{\Lm_r}$ is equal to $(+,+,\cdots,+)$. Using (\ref{Dsp-f1}) and (\ref{Dsp-f2}), we have
\begin{eqnarray*}
\ovl{u_{\leq k}}\cdot u_{\Lm_r}&=&(\ovl{s_1}\ \ovl{s_2}\cdots\ovl{s_r})^{m-1}(+,+,+,\cdots,+) \\
&=&(\ovl{s_1}\ \ovl{s_2}\cdots\ovl{s_r})^{m-2}(-,+,+,\cdots,-) \\
&\vdots&\\
&=&\begin{cases}(\underbrace{-,-,\cdots,-}_{m-1},+,\cdots,+,-) & {\rm if}\ m\ {\rm is\ even}, \\ 
(\underbrace{-,-,\cdots,-}_{m-1},+,\cdots,+,+) & {\rm if}\ m\ {\rm is\ odd}. \\ 
\end{cases}
\end{eqnarray*}
Now, we define the following notation: For $(\ep_1,\cdots,\ep_r)\in{\mathbf B}_{{\rm sp}}^{(r)}$, let $\{k_1,\cdots,k_t\}$ be the set of indices such that $\{1\leq i\leq r|\ \ep_i=- \}=\{k_1,\cdots,k_t\}$. Then we denote $[k_1,\cdots,k_t]:=(\ep_1,\cdots,\ep_r)$. Applying $x_{-1}(Y_{s,1})\cdots x_{-r}(Y_{s,r})$ on $[k_1,\cdots,k_t]$ $(1\leq s\leq m)$, it becomes a linear combination of $\{[j_1,\cdots,j_t]|\ k_{l-1}<j_l\leq k_l\ (1\leq l\leq t)\}\cup\{[j_1,\cdots,j_t,j_{t+1},r]|k_{l-1}<j_l\leq k_l\ (1\leq l\leq t+1)\}$. It follows from (\ref{Bsp-f0}) and (\ref{Bsp-f1}) that
\begin{multline*}
x_{-k_{l-1}}(Y_{s,k_{l-1}})\cdots x_{-(k_l-1)}(Y_{s,k_l-1})[\cdots,k_l,\cdots]=\cdots+\left(\frac{1}{Y_{s,k_l-1}}\right)^{\frac{1}{2}}\left(\frac{Y_{s,k_l-1}}{Y_{s,k_l-2}}\right)^{\frac{1}{2}}\cdots 
\left(\frac{Y_{s,j_l}}{Y_{s,j_l-1}}\right)^{\frac{1}{2}} \left(\frac{Y_{s,j_l-2}}{Y_{s,j_l-1}}\right)^{\frac{1}{2}}
\\
\left(\frac{Y_{s,j_l-3}}{Y_{s,j_l-2}}\right)^{\frac{1}{2}}
 \cdots \left(\frac{Y_{s,k_{l-1}+1}}{Y_{s,k_{l-1}+2}}\right)^{\frac{1}{2}}\frac{Y_{s,k_{l-1}}}{(Y_{s,k_{l-1}+1})^{\frac{1}{2}}}[\cdots,j_l,\cdots]+\cdots =\cdots+\frac{Y_{s,k_{l-1}}}{Y_{s,j_l-1}}[\cdots,j_l,\cdots]+\cdots.
\end{multline*}
Thus, the coefficient of $[j_1,\cdots,j_t]$ in $x_{-1}(Y_{s,1})\cdots x_{-r}(Y_{s,r})[k_1,\cdots,k_t]$ is
\[
\begin{cases} \frac{Y_{s,k_1}}{Y_{s,j_1-1}}\frac{Y_{s,k_2}}{Y_{s,j_2-1}}\cdots \frac{Y_{s,k_t}}{Y_{s,j_t-1}}\cdot \frac{1}{Y_{s,r}} & {\rm if}\ k_{t-1}, k_{t}\neq r-1, \\
\frac{Y_{s,k_1}}{Y_{s,j_1-1}}\frac{Y_{s,k_2}}{Y_{s,j_2-1}}\cdots \frac{Y_{s,k_t}}{Y_{s,j_t-1}} & {\rm if}\ k_{t-1}\ {\rm or}\ k_{t}=r-1. \\
\end{cases}
 \]
and the one of $[j_1,\cdots,j_t,j_{t+1},r]$ is
\[ \frac{Y_{s,k_1}}{Y_{s,j_1-1}}\frac{Y_{s,k_2}}{Y_{s,j_2-1}}\cdots \frac{Y_{s,k_t}}{Y_{s,j_t-1}}\frac{1}{Y_{s,j_{t+1}-1}}, \]
which coincide with $Q({\rm vt}(m-s;k_1,\cdots,k_t)\rightarrow{\rm vt}(m-s+1;j_1,\cdots,j_{t'}))$ in Definition \ref{Dlabeldef-sp}. In our notation, we can write $[\phi]:=(+,+,\cdots,+)$, $[1,2,\cdots,m-1]=(\underbrace{-,-,\cdots,-}_{m-1},+,\cdots,+)$ and $[1,2,\cdots,m-1,r]=(\underbrace{-,-,\cdots,-}_{m-1},+,\cdots,+-)$. Therefore, we get (\ref{Dpathlem-spin}). We can similarly verify the case $i_k=r-1$ by the same way as the case $i_k=r$. 

\end{proof}

[{\sl Proof of Theorem \ref{thmd2}.}]

We take a path in $X^{r}(m,m-1)$ in the form $p={\rm vt}(m;\phi)\rightarrow {\rm vt}(m-1;k^{(m-1)}_1,\cdots,k^{(m-1)}_{t_{m-1}})\rightarrow {\rm vt}(m-2;k^{(m-2)}_1,\cdots,k^{(m-2)}_{t_{m-2}})\rightarrow \cdots \rightarrow {\rm vt}(1;k^{(1)}_1,\cdots,k^{(1)}_{t_{1}})\rightarrow {\rm vt}(0;k^{(0)}_1,\cdots,k^{(0)}_{t_{0}})$, where $t_{m-s-1}=t_{m-s}$ or $t_{m-s}+2$ $(0\leq s\leq m-1,\ t_m:=0)$. There exists a unique even number $s$ $(0\leq s\leq m-1)$ such that $t_m=0$, $t_{m-1}=t_{m-2}=2$, $t_{m-3}=t_{m-4}=4,\ \cdots,\ t_{m-s}=s,\ t_{m-s-1}=s$. We also see that $k^{(m-s-1)}_s=\cdots=k^{(0)}_s=s$ $(1\leq s\leq m-1)$ by Definition \ref{Dpathdef-sp} (ii), (iii) and (v). Then, similar to Lemma \ref{kpro2d}, we can verify
\[ Q(p)=D(m-1,\ovl{k^{(m-1)}_1})D(m-2,\ovl{k^{(m-2)}_2})\cdots D(m-s,\ovl{k^{(m-s)}_s})D(m-s,r+1) \]
and we obtain Theorem \ref{thmd2} in the same way as the proof of Theorem \ref{thmd1}. \qed


\section{The proof of Theorem \ref{thm1}}\label{prsec}

In this section, we shall give the proof of Theorem \ref{thm1}. We have already proven Theorem \ref{thm1} for the simple algebraic group of type ${\rm A}_r$ in \cite{KaN:2015}. So, we shall prove the theorem for other type ${\rm B}_r$, ${\rm C}_r$ and ${\rm D}_r$. First, let us review the theorem for type ${\rm A}_r$. For $k\in\mathbb{Z}$ $(1\leq k\leq r)$, we set $|k|=|\ovl{k}|=k$ and $[1,k]:=\{1,2,\cdots,k\}$.

\subsection{Type ${\rm A}_r$}

For $1\leq l\leq m$ and $1\leq k\leq r$, we set the Laurent monomial $A(l,k):=
\frac{Y_{l,k-1}}{Y_{l,k}}$.

\begin{lem}\cite{KaN:2015}
In the setting of Theorem \ref{thm1}, we have
\begin{equation}\label{A-explicit}
\Delta^L(k;{\rm \bf{i}})=\sum_{(*)} A(m-k_1+1,k_1)A(m-k_2+2,k_2)\cdots A(m-k_d+d,k_d),
\end{equation}
where $(*)$ is the condition for $\{k_i\}_{1\leq i\leq d}$ $:$\q $1\leq k_1<\cdots<k_d\leq m-1+d.$
\end{lem}

\begin{thm}\cite{KaN:2015}
\[
\Delta^L(k;{\rm \bf{i}})(\textbf{Y})= \sum_{b\in B^-(\Lm_d)_{u_{\leq k}}}\ \mu(b),
\]
where $B^-(\Lm_d)_{u_{\leq k}}$ is the lower Demazure crystal of type ${\rm A}_r$, and $\mu:B(\Lm_d)\rightarrow \cY$ is the monomial realization of the crystal base $B(\Lm_d)$ such that $\mu(v_{\Lm_d})=\frac{1}{Y_{m,d}}$, where $v_{\Lm_d}$ is the lowest weight vector in $B(\Lm_d)$.
\end{thm}

In Theorem \ref{thm1} (\ref{GenDem}), we have used the notation $a_b\in \mathbb{Z}_{>0}$. In the case of type ${\rm A}_r$, we find that $a_b\equiv 1$ for all $b\in B^-(\Lm_d)_{u_{\leq k}}$.

\subsection{Type ${\rm C}_r$}

We treat the type ${\rm C}_r$. In \cite{KaN2:2016}, we calculate the explicit formula of the generalized minors $\Delta^L(k;{\rm \bf{i}})$ on the simple algebraic group of type ${\rm C}_r$ and ${\rm \bf{i}}$ in (\ref{redwords2}). For $l\in[1,m]$ and $k\in J_{{\rm C}}=\{i,\ovl{i}|i=1,2,\cdots,r\}$, we set the Laurent monomials
\begin{equation}\label{ccbar}
C(l,k):=
\begin{cases}
\frac{Y_{l,k-1}}{Y_{l,k}} & {\rm if}\ 1\leq k\leq r, \\
\frac{Y_{l,|k|}}{Y_{l+1,|k|-1}} & {\rm if}\ \ovl{r}\leq k\leq \ovl{1}.
\end{cases}
\end{equation}

\begin{lem}\cite{KaN2:2016}\label{C-exlem}
In the setting of Theorem \ref{thm1}, we have
\begin{equation}\label{C-explicit}
\Delta^L(k;{\rm \bf{i}})=\sum_{(*)} C(l_1,k_1)C(l_2,k_2)\cdots C(l_d,k_d),
\end{equation}
\begin{equation}\label{Lidef}
l_i:=
\begin{cases}
m-k_i+i & {\rm if}\ 1\leq k_i\leq r, \\
m-r-1+i & {\rm if}\ \ovl{r}\leq k_i\leq \ovl{1}.
\end{cases}
\end{equation}
where $(*)$ is the condition for $k_i$ $(1\leq i\leq d)$ $:$ 
\begin{eqnarray}
& &1\leq k_1<k_2<\cdots<k_d\leq\ovl{1}, \label{C-cond1}\\
& &k_i\leq m-1+i\qq ({\rm if}\ 1\leq i\leq r-m+1), \label{C-cond2}\\
& &k_i\leq \ovl{d-i+1}\qq ({\rm if}\ r-m+2\leq i\leq d). \label{C-cond3}
\end{eqnarray}
\end{lem}

\begin{ex}\label{monorealex3}

Let us consider the case of type ${\rm C}_2$, we set $u=s_1s_2s_1s_2\in W$ and its reduced word ${\rm \bf{i}}=(1,2,1,2)$, which is the same setting in Example \ref{monorealex2}. Then $m=2$, $i_1=i_3=1$, $i_2=i_4=2$. Following Lemma \ref{C-exlem}, we can calculate the generalized minor $\Delta^L(2;{\rm \bf{i}})(\textbf{Y})$:
\begin{eqnarray*}
\Delta^L(2;{\rm \bf{i}})(\textbf{Y})&=&
C(2,1)C(2,2)+C(2,1)C(1,\ovl{2})+C(2,1)C(1,\ovl{1})+C(1,2)C(1,\ovl{2})+C(1,2)C(1,\ovl{1}) \\
&=&\frac{1}{Y_{2,2}}+\frac{Y_{1,2}}{Y^2_{2,1}}+2\frac{Y_{1,1}}{Y_{2,1}}+\frac{Y^2_{1,1}}{Y_{1,2}},
\end{eqnarray*}
which coincides with the one in $(\ref{mono2exp})$. Unlike the case of type ${\rm A}_r$, there exists a term with coefficient $2$. In the cases of type ${\rm B}_r$, ${\rm C}_r$, and ${\rm D}_r$, some terms in $\Delta^L(k;{\rm \bf{i}})(\textbf{Y})$ have coefficients greater than $1$.
\end{ex}

\begin{rem}\label{remC0}

\begin{itemize} 
\item[(i)]We use different notation from the one in \cite{KaN2:2016}. It was defined $\ovl{C}(l,k):=\frac{Y_{l,k-1}}{Y_{l,k}}$ and $C(l,k):=\frac{Y_{l,k+1}}{Y_{l+1,k}}$ in \cite{KaN2:2016}, which corresponds to $C(l,k)$ and $C(l,\ovl{k+1})$ in $(\ref{ccbar})$.
\item[(ii)] In the case $d\leq r-m+1$, the generalized minor $\Delta^L(k;{\rm \bf{i}})$ in $(\ref{C-explicit})$ coincides with the one in $(\ref{A-explicit})$ for type ${\rm A}_r$. Further, in this case, we can verify Theorem \ref{thm1} by the same way as in \cite{KaN:2015}. 
Therefore, in the rest of paper, we shall show for $d>r-m+1$.
\end{itemize}

\end{rem}

Set the subset $\mathbb{B}\subset\cY$ as
\begin{equation}\label{SubsetB}
\mathbb{B}:=\{C(l_1,k_1)C(l_2,k_2)\cdots C(l_d,k_d) | \{k_i\}\ {\rm satisfies}\ (\ref{C-cond1}),\ (\ref{C-cond2})\ {\rm and}\ (\ref{C-cond3}) \}, 
\end{equation}
where above $\{l_i\}$ is as in (\ref{Lidef}).
\begin{rem}\label{remC1}
The definition $(\ref{SubsetB})$ implies that $X=C(l_1,k_1)C(l_2,k_2)\cdots C(l_d,k_d)\in\mathbb{B}$ if and only if

\begin{itemize} 
\item[(i)] $1\leq k_1<\cdots<k_d\leq \ovl{1}$\ $(\Leftrightarrow(\ref{C-cond1}))$,
\item[(ii)] The monomial $X$ has at most $d-r+m-1$ factors in the form $C(s,b)$ $(s\in\mathbb{Z},\ b\in\{\ovl{r},\cdots,\ovl{1}\})$. 
\end{itemize}

The conditions $k_i\leq m-1+i$ $(1\leq i\leq r-m+1)$ and $k_i\leq \ovl{d-i+1}$ $(r-m+2\leq i\leq d)$ in Lemma $\ref{C-exlem}$ follows from $(i)$ and $(ii)$.
\end{rem}

\begin{lem}\label{Clem1}
We take $X:=C(l_1,k_1)C(l_2,k_2)\cdots C(l_d,k_d)\in\mathbb{B}$.

\begin{itemize} 
\item[(i)]For $k\in[1,r-1]$, we have $\tilde{e}_k C(l,k)=A_{l-1,k}\cdot C(l,k)=C(l-1,k+1)$ and 
$\tilde{e}_k C(l,\ovl{k+1})=A_{l,k}\cdot C(l,\ovl{k+1})=C(l,\ovl{k})$ in $\cY$. Furthermore, $\tilde{e}_r C(l,r)=A_{l-1.r}\cdot C(l,r)=C(l-1,\ovl{r})$.
\item[(ii)] If $k\in[1,r-1]$ and $\tilde{e}_k X\neq0$, then there exists $j\in[1,d]$ such that the following $(a)$ or $(b)$ holds:

$(a)$ $k_j=k$, $k_{j+1}>k+1$ and
\begin{equation}\label{ekXc1}
 \tilde{e}_k X=C(l_1,k_1)\cdots C(l_{j-1},k_{j-1})C(l_j-1,k+1)C(l_{j+1},k_{j+1})\cdots C(l_d,k_d).  
\end{equation}
$(b)$ $k_j=\ovl{k+1}$, $k_{j+1}>\ovl{k}$ and
\begin{equation}\label{ekXc2}
\tilde{e}_k X=C(l_1,k_1)\cdots C(l_{j-1},k_{j-1})C(l_j,\ovl{k})C(l_{j+1},k_{j+1})\cdots C(l_d,k_d).  
\end{equation}
\item[(iii)]
If $\tilde{e}_r X\neq0$, then there exists $j$ $(1\leq j\leq d)$ such that $k_j=r$, $k_{j+1}>\ovl{r}$ and
\begin{equation}\label{ekXc3}
\tilde{e}_r X=C(l_1,k_1)\cdots C(l_{j-1},k_{j-1})C(l_j-1,\ovl{r})C(l_{j+1},k_{j+1})\cdots C(l_d,k_d).
\end{equation}
Furthermore, we have $\tilde{e}^2_r X=0$.
\item[(iv)]For $k\in[1,r]$, suppose that $\tilde{e}_k X\neq0$. Then, the monomial $\tilde{e}_k X$ satisfies the condition in Remark \ref{remC1} $(i)$, that is, we can write $\tilde{e}_k X=C(l'_1,k'_1)\cdots C(l'_d,k'_d)$ with some $1\leq k'_1<\cdots<k'_d\leq \ovl{1}$.
\end{itemize}

\end{lem}

\begin{proof}
(i) If $1\leq k\leq r-1$, it follows from the definition of $\tilde{e}_k$ in $\cY$ that $\tilde{e}_k C(l,k)=A_{l-1,k}\cdot C(l,k)=\frac{Y_{l-1,k}Y_{l,k}}{Y_{l-1,k+1}Y_{l,k-1}}\cdot \frac{Y_{l,k-1}}{Y_{l,k}}=C(l-1,k+1)$. We can verify the remained cases in the similar way.

(ii), (iii) For $k\in[1,r-1]$, we suppose that $\tilde{e}_k X\neq0$. If the Laurent monomial $X$ does not include factors in the form $Y^{-1}_{*,k}$, then we get $\varepsilon_k(X)=0$ and $\tilde{e}_k X=0$, which is absurd. Therefore, since $k_1<\cdots<k_d$, the monomial $X$ includes either one or two factors in the form $Y^{-1}_{*,k}$, which means that $X$ has the factors in the form either $C(l_j,k)$ or $C(l_{j'}, \ovl{k+1})$, or both (see (\ref{ccbar})). If $X$ does not have the factor $Y^{-1}_{l_j,k}$ then $X$ has the factor $Y^{-1}_{l_{j'}+1,k}$, $k_{j'}=\ovl{k+1}$, and $k_{j'+1}>\ovl{k}$. Since $\tilde{e}_k X\neq0$, we have $n_{e_k}=l_{j'}$ and hence $\tilde{e}_k X=A_{l_{j'},k}X=C(l_1,k_1)\cdots C(l_{j-1},k_{j-1})C(l_j,\ovl{k})C(l_{j+1},k_{j+1})\cdots C(l_d,k_d)$. Similarly, if $X$ does not have the factor $Y^{-1}_{l_{j'}+1,k}$ then we have $k_j=k$, $k_{j+1}>k+1$, and $\tilde{e}_k X=A_{l_{j}-1,k}X=C(l_1,k_1)\cdots C(l_{j-1},k_{j-1})C(l_j-1,k+1)C(l_{j+1},k_{j+1})\cdots C(l_d,k_d)$. If $X$ includes both factors $Y^{-1}_{l_j,k}$ and $Y^{-1}_{l_{j'}+1,k}$, then $n_{e_k}=l_{j'}$ or $n_{e_k}=l_{j}-1$, which implies that $\tilde{e}_k X$ is in the form (\ref{ekXc1}) or (\ref{ekXc2}). Arguing similarly, we see that if $\tilde{e}_r X\neq0$ then $\tilde{e}_r X$ is in the form (\ref{ekXc3}). Then, since $k_1<k_2<\cdots<k_j=r<k_{j+1}<\cdots<k_d$, the Laurent monomial $\tilde{e}_r X$ has no factors in the form $C(*,r)$, which means that $\tilde{e}_r X$ does not have factors in the form $Y^{-1}_{*,r}$. Hence, it follows from $\varepsilon_r(\tilde{e}_r X)=0$ that $\tilde{e}^2_r X=0$.

(iv) The explicit forms (\ref{ekXc1}), (\ref{ekXc2}) and (\ref{ekXc3}) with each condition for $k_{j+1}$ imply that the monomial $\tilde{e}_k X$ satisfies the condition in Remark \ref{remC1} (i). 
\end{proof}

\begin{lem}\label{Clem2}
For $X=C(l_1,k_1)C(l_2,k_2)\cdots C(l_d,k_d)\in\mathbb{B}$ and $s\in\{0,1,\cdots,r-m+1\}$, we suppose that $1\leq k_1<\cdots<k_{d-s}\leq r$ and $\ovl{r}\leq k_{d-s+1}<\cdots<k_{d-1}<k_d\leq\ovl{1}$, that is, $s:=\#\{i\in [1,d]|k_i\in\{\ovl{r},\cdots,\ovl{1}\}\}$. Then there exist non-negative integers $\{a(i,j)\}_{1\leq i\leq r-d+s,\ 1\leq j\leq r}$ such that
\begin{equation}\label{Clem2-1}
X=\tilde{e}^{a(r-d+s,1)}_1\cdots\tilde{e}^{a(r-d+s,r)}_r\cdots \tilde{e}^{a(2,1)}_1\cdots\tilde{e}^{a(2,r)}_r
 \tilde{e}^{a(1,1)}_1\cdots\tilde{e}^{a(1,d)}_d \frac{1}{Y_{m,d}}. 
\end{equation}
\end{lem}

\begin{proof}

We use induction on $s$. In the case $s=0$, we obtain $1\leq k_1<\cdots<k_d\leq r$. Let us put $a(1,i)=a(2,i+1)=\cdots=a(k_i-i,k_{i}-1)=1$ for $1\leq i\leq d$ and $a(i,j)=0$ for other $(i,j)$. Using Lemma \ref{Clem1} (i), we see that $\tilde{e}_{k_i-1}\cdots\tilde{e}_{i+2}\tilde{e}_{i+1}\tilde{e}_i C(m,i)=C(m-k_i+i,k_i)=C(l_i.k_i)$ $(i=1,2,\cdots,d)$. Therefore, by the action of $\tilde{e}^{a(r-d,1)}_1\cdots\tilde{e}^{a(r-d,r)}_r\cdots \tilde{e}^{a(2,1)}_1\cdots\tilde{e}^{a(2,r)}_r
 \tilde{e}^{a(1,1)}_1\cdots\tilde{e}^{a(1,d)}_d$, each factor $C(m,i)$ in $C(m,1)C(m,2)\cdots C(m,d)=\frac{1}{Y_{m,d}}$ becomes $C(m-k_i+i,k_i)=C(l_i.k_i)$, which means
\[
\tilde{e}^{a(r-d,1)}_1\cdots\tilde{e}^{a(r-d,r)}_r\cdots \tilde{e}^{a(2,1)}_1\cdots\tilde{e}^{a(2,r)}_r
 \tilde{e}^{a(1,1)}_1\cdots\tilde{e}^{a(1,d)}_d \frac{1}{Y_{m,d}}
 =C(l_1,k_1)C(l_2,k_2)\cdots C(l_d,k_d)=X. 
\]
Next, we consider the case $s>0$. The explicit form of $C(l,k)$ in (\ref{ccbar}) implies that
\begin{eqnarray*}
& & \prod^{k_1-1}_{i=1}C(m-i,\ovl{i}) \prod^{k_2-1}_{i=k_1+1}C(m-i+1,\ovl{i})
\prod^{k_3-1}_{i=k_2+1}C(m-i+2,\ovl{i})\cdot \\
& &\qq \qq \cdots \prod^{k_{d-s}-1}_{i=k_{d-s-1}+1}C(m-i+d-s-1,\ovl{i}) \prod^{r}_{i=k_{d-s}+1}C(m-i+d-s,\ovl{i}) \\
&=& \frac{Y_{m-k_1+1,k_1-1}}{Y_{m-k_1+1,k_1}}\cdot \frac{Y_{m-k_2+2,k_2-1}}{Y_{m-k_2+2,k_2}}
\cdots \frac{Y_{m-k_{d-s}+d-s,k_{d-s}-1}}{Y_{m-k_{d-s}+d-s,k_{d-s}}}\cdot Y_{m-r+d-s,r}\\
&=& C(l_1,k_1)C(l_2,k_2)\cdots C(l_{d-s},k_{d-s})\cdot Y_{m-r+d-s,r}.
\end{eqnarray*}
Hence, we obtain
\begin{multline}\label{Clem2-2} C(l_1,k_1)C(l_2,k_2)\cdots C(l_{d-s},k_{d-s})= \\
C(m-1,\ovl{k'_1})C(m-2,\ovl{k'_2})\cdots C(m-r+d-s,\ovl{k'_{r-d+s}})\cdot\frac{1}{Y_{m-r+d-s,r}}, 
\end{multline}
where $\{k'_1,\ k'_2,\ \cdots,\ k'_{r-d+s} \}:=\{1,2,\cdots,r\}\setminus\{k_1,k_2,\cdots,k_{d-s} \}$ and $k'_1<k'_2<\cdots<k'_{r-d+s}$. Thus, the Laurent monomial $X$ can be written as follows:
\begin{multline}\label{Clem2-3}
X=C(m-1,\ovl{k'_1})C(m-2,\ovl{k'_2})\cdots C(m-r+d-s,\ovl{k'_{r-d+s}})\cdot\frac{1}{Y_{m-r+d-s,r}}\\
\cdot C(m-r+d-s,k_{d-s+1})\cdots C(m-r+d-2,k_{d-1})C(m-r+d-1,k_d).
\end{multline}

Setting $\kappa=|k_{d-s+1}|$, we suppose that $k'_{j-1}<\kappa\leq k'_{j}$ for some $j$ $(1\leq j\leq r-d+s)$. Then the Laurent monomial $X$ has the factor $Y_{m-r+d-s,\kappa}$ and does not have a factor in the form $Y^{-1}_{a,*}$ $(a\leq m-r+d-s)$. It follows from the definition (\ref{wtph}) of $\varphi_{\kappa}$ that $\varphi_{\kappa}(X)>0$. Furthermore, as we have seen in Lemma \ref{Clem1} (i), if $\kappa<k'_{j}$ (resp. $\kappa=k'_{j}$), then the factor $C(m-r+d-s,\ovl{\kappa})$ in $X$ is changed to $C(m-r+d-s,\ovl{\kappa+1})$ by the action of $\tilde{f}_{\kappa}$ (resp. $\tilde{f}^2_{\kappa}$). Repeating this argument, we see that by the action of
\begin{multline*}
F_0=(\tilde{f}_{r}\tilde{f}^2_{r-1}\cdots\tilde{f}^2_{k'_{r-d+s}}\tilde{f}_{k'_{r-d+s}-1})(\tilde{f}^2_{k'_{r-d+s}-2}\cdots \tilde{f}^2_{k'_{r-d+s-1}}\tilde{f}_{k'_{r-d+s-1}-1})\\
\cdots(\tilde{f}^2_{k'_{j+2}-2}\cdots\tilde{f}^2_{k'_{j+1}}
\tilde{f}_{k'_{j+1}-1})(\tilde{f}^2_{k'_{j+1}-2}\cdots\tilde{f}^2_{k'_{j}+1}\tilde{f}^2_{k'_{j}}\tilde{f}_{k'_{j}-1}
\cdots\tilde{f}_{\kappa+1}\tilde{f}_{\kappa}),
\end{multline*}
each factor $C(m-i,\ovl{k'_{i}})$ is sent to $C(m-i,\ovl{k'_{i+1}-1})$ $(j\leq i\leq r-d+s-1)$. And we find that $C(m-r+d-s,\ovl{k'_{r-d+s}})$ and $C(m-r+d-s,k_{d-s+1})=C(m-r+d-s,\ovl{\kappa})$ become $C(m-r+d-s,\ovl{r})$ and $C(m-r+d-s+1,r)$ respectively. Hence, 
\begin{multline*}
 F_0\cdot X=C(m-1,\ovl{k'_1})\cdots C(m-j+1,\ovl{k'_{j-1}})C(m-j,\ovl{k'_{j+1}-1})\cdots \\
C(m-r+d-s+1,\ovl{k'_{r-d+s}-1})C(m-r+d-s,\ovl{r})\cdot\frac{1}{Y_{m-r+d-s,r}}\\
\cdot C(m-r+d-s+1,r)C(m-r+d-s+1,k_{d-s+2})\cdots C(m-r+d-2,k_{d-1})C(m-r+d-1,k_d).
\end{multline*}
Similar to (\ref{Clem2-2}), setting $\{k''_{1},k''_{2},\cdots,k''_{d-s}\}:=\{1,2,\cdots,r\}\setminus\{k'_1,\cdots,k'_{j-1},k'_{j+1}-1,k'_{j+2}-1,\cdots,k'_{r-d+s}-1,r\}$ with $1\leq k''_1<\cdots<k''_{d-s}<r$, we have
\begin{multline*}
 F_0\cdot X=C(l''_1,k''_1)C(l''_2,k''_2)\cdots C(l''_{d-s},k''_{d-s})C(m-r+d-s+1,r)\\
\cdot C(m-r+d-s+1,k_{d-s+2}) \cdots C(m-r+d-2,k_{d-1})C(m-r+d-1,k_d),
\end{multline*}
where $l''_i:=m-k''_i-i$. By $1\leq k''_1<\cdots<k''_{d-s}<r<\ovl{r}\leq k_{d-s+2}<\cdots<k_{d-1}<k_d\leq\ovl{1}$, the monomial $F_0\cdot X$ has the $s-1$ factors in the form $C(*,b)$ $(b\in\{\ovl{r},\cdots,\ovl{1}\})$ and 
belongs to $\mathbb{B}$ (see Remark \ref{remC1}). Hence using the assumption of induction for $F_0\cdot X$, we get
\[ F_0\cdot X=\tilde{e}^{a(r-d+s-1,1)}_1\cdots\tilde{e}^{a(r-d+s-1,r)}_r\cdots\tilde{e}^{a(1,1)}_1\cdots\tilde{e}^{a(1,d)}_d\frac{1}{Y_{m,d}}, \]
for some non-negative integers $\{a(i,j)\}$. 

In general, for each operator $F:=\tilde{f}_{i_1}\tilde{f}_{i_2}\cdots \tilde{f}_{i_n}$ on $\cY$, we denote the operator $F^{*}$ by 
\begin{equation}\label{Finv}
F^{*}:=\tilde{e}_{i_n}\cdots\tilde{e}_{i_2}\tilde{e}_{i_1},
\end{equation}
 which satisfies $F^{*}\cdot F\cdot X=X$ if $F\cdot X\neq0$. Following (\ref{Finv}), we set $F_0^{*}:=\tilde{e}_{\kappa}\tilde{e}_{\kappa+1}\tilde{e}_{k'_j-1}\tilde{e}^2_{k'_j}\cdots \tilde{e}^2_{r-1}\tilde{e}_r$. Then
\[ X=F_0^{*}\tilde{e}^{a(r-d+s-1,1)}_1\cdots\tilde{e}^{a(r-d+s-1,r)}_r\cdots\tilde{e}^{a(1,1)}_1\cdots\tilde{e}^{a(1,d)}_d\frac{1}{Y_{m,d}},  \]
which is the same form as in (\ref{Clem2-1}). 

Similar to this, we can verify (\ref{Clem2-1}) in the case $k'_{r-d+s}<k_{d-s+1}=\ovl{\kappa}$. We put $\kappa'=k'_{r-d+s}$ and suppose that $|k_{j}|<\kappa'\leq|k_{j-1}|$ for some $d-s+2\leq j\leq d$. Using the operator
\begin{multline*}
F'_0=(\tilde{f}_{r}\tilde{f}^2_{r-1}\cdots\tilde{f}^2_{|k_{d-s+1}|}\tilde{f}_{|k_{d-s+1}|-1})(\tilde{f}^2_{|k_{d-s+1}|-2}\cdots \tilde{f}^2_{|k_{d-s+2}|}\tilde{f}_{|k_{d-s+2}|-1})\\
\cdots(\tilde{f}^2_{|k_{j-3}|-2}\cdots\tilde{f}^2_{|k_{j-2}|}
\tilde{f}_{|k_{j-2}|-1})(\tilde{f}^2_{|k_{j-2}|-2}\cdots\tilde{f}^2_{|k_{j-1}|+1}\tilde{f}^2_{|k_{j-1}|}\tilde{f}_{|k_{j-1}|-1}
\cdots\tilde{f}_{\kappa'+1}\tilde{f}_{\kappa'}),
\end{multline*}
instead of $F_0$, we get (\ref{Clem2-1}) by the same way as above. 

\end{proof}

[{\sl Proof of Theorem \ref{thm1} for type ${\rm C}_r$}\ ]

All we have to show is $\mu(B^-(\Lm_d)_{u_{\leq k}})=\mathbb{B}$, where $\mathbb{B}$ is as in (\ref{SubsetB}). First, let us prove the inclusion $\mu(B^-(\Lm_d)_{u_{\leq k}})\subset\mathbb{B}$. Recall that the crystal $B(\Lm_d)$ has the lowest weight vector $v_{\Lm_d}$, and
\[
(\underbrace{1,2,\cdots,r}_{1{\rm st\ cycle}},\underbrace{1,2,\cdots,r}_{2{\rm nd\ cycle}}\cdots \underbrace{1,2,\cdots,r}_{m-2\ {\rm th\ cycle}},\underbrace{1,2,\cdots,d}_{m-1\ {\rm th\ cycle}}) 
\]
is a reduced word of $u_{\leq k}$ (see (\ref{redwords2}) in Sect.\ref{gmc}). Since we have $\mu(v_{\Lm_d})=C(m,1)C(m,2)\cdots C(m,d)=\frac{1}{Y_{m,d}}\in \mathbb{B}$ and Theorem \ref{kashidem}, we need to see 
\begin{equation}\label{Cproof1}
\tilde{e}^{a(1,1)}_1\cdots\tilde{e}^{a(1,r)}_r \tilde{e}^{a(2,1)}_1\cdots\tilde{e}^{a(2,r)}_r\cdots \tilde{e}^{a(m-1,1)}_1\cdots\tilde{e}^{a(m-1,d)}_d \frac{1}{Y_{m,d}} \in \mathbb{B}\cup\{0\}, 
\end{equation}
for any integers $a(i,j)\in \mathbb{Z}_{\geq0}$. Fixing these integers $a(i,j)$, we define the operators on $\cY$ as $E_{m-1}:=\tilde{e}^{a(m-1,1)}_1\cdots\tilde{e}^{a(m-1,d)}_d$ and $E_{m-s}:=\tilde{e}^{a(m-s,1)}_1\cdots\tilde{e}^{a(m-s,r)}_r$ for $2\leq s\leq m-1$. It follows from Lemma \ref{Clem1} (ii) that 
\[E_{m-1}\frac{1}{Y_{m,d}}=\tilde{e}^{a(m-1,1)}_1\cdots\tilde{e}^{a(m-1,d)}_d \frac{1}{Y_{m,d}}\]
is either a product of $\{C(*,b)|1\leq b\leq d+1,\ *\ {\rm is\ an\ integer}\}$ belonging to $\mathbb{B}$ or $0$. Similarly, $E_{m-2}E_{m-1} \frac{1}{Y_{m,d}}$ is either a product of $\{C(*,b)|1\leq b\leq d+2\}$ or $0$. Repeating this argument, we see that $E_{m-r+d}\cdots E_{m-2}E_{m-1}\frac{1}{Y_{m,d}}$
is either a product of $\{C(*,b)|1\leq b\leq r\}$ belonging to $\mathbb{B}$ or $0$. Hence,
\[ E_{m-r+d}\cdots E_{m-2}E_{m-1}\frac{1}{Y_{m,d}}=C(l_1,k_1)C(l_2,k_2)\cdots C(l_d,k_d), \]
for some $1\leq k_1<k_2<\cdots<k_d\leq r$. Using Lemma \ref{Clem1} (ii), (iii) and (iv), we see that $E_{m-r+d-1}E_{m-r+d}\cdots E_{m-2}E_{m-1}\frac{1}{Y_{m,d}}$ is either $0$ or a product of $\{C(*,b)| b\in J_{{\rm C}} \}$ such that the number of product of $\{C(*,b')|b'\in\{\ovl{r},\ovl{r-1},\cdots,\ovl{1}\}\}$ is at most $1$, and it satisfies the condition in Remark \ref{remC1} (i). Since we have supposed $1\leq d-r+m-1$ in Remark
\ref{remC0} (ii), it also satisfies the condition in Remark \ref{remC1} (ii), which means that it belongs to $\mathbb{B}\cup\{0\}$. Repeating this argument, we can verify that $E_{1}E_{2}\cdots E_{m-2}E_{m-1}\frac{1}{Y_{m,d}}$ is either $0$ or a product of $\{C(*,b)|b\in J_{{\rm C}}\}$ such that the number of product of $\{C(*,b')|b'\in\{\ovl{r},\ovl{r-1},\cdots,\ovl{1}\}\}$ is at most $d-r+m-1$
, and it satisfies the condition in Remark \ref{remC1} (i). Hence,  $E_{1}E_{2}\cdots E_{m-2}E_{m-1}\frac{1}{Y_{m,d}}\in\mathbb{B}\cup\{0\}$, which means (\ref{Cproof1}).

Finally, we shall prove the inverse inclusion $\mu(B^-(\Lm_d)_{u_{\leq k}})\supset\mathbb{B}$. Let $X=C(l_1,k_1)C(l_2,k_2)\cdots C(l_d,k_d)\in\mathbb{B}$ be a monomial. By Theorem \ref{kashidem}, we need to show that there exist non-negative integers $\{a(i,j)\}$ such that
\begin{equation}\label{Cproof2}
X=\tilde{e}^{a(1,1)}_1\cdots\tilde{e}^{a(1,r)}_r \tilde{e}^{a(2,1)}_1\cdots\tilde{e}^{a(2,r)}_r\cdots \tilde{e}^{a(m-1,1)}_1\cdots\tilde{e}^{a(m-1,d)}_d \frac{1}{Y_{m,d}},
\end{equation}
which has been already obtained in Lemma \ref{Clem2}. Thus, we have $\mu(B^-(\Lm_d)_{u_{\leq k}})\supset\mathbb{B}$. \qed

\subsection{Type ${\rm B}_r$}

Next, we treat the case of type ${\rm B}_r$. For $l\in[1,m]$ and $k\in J_{{\rm B}}$, we set the Laurent monomials
\begin{equation}\label{bbbar}
B(l,k):=
\begin{cases}
\frac{Y_{l,k-1}}{Y_{l,k}} & {\rm if}\ 1\leq k\leq r-1, \\
\frac{Y_{l,r-1}}{Y^2_{l,r}} & {\rm if}\ k= r, \\
\frac{Y_{l,r}}{Y_{l+1,r}} & {\rm if}\ k= 0, \\
\frac{Y^2_{l,r}}{Y_{l+1,r-1}} & {\rm if}\ k= \ovl{r}, \\
\frac{Y_{l,|k|}}{Y_{l+1,|k|-1}} & {\rm if}\ \ovl{r-1}\leq k\leq \ovl{1}.
\end{cases}
\end{equation}

In the case $i_k=d<r$, the minors $\Delta^L(k;{\rm \bf{i}})$ are explicitly described as follows.

\begin{lem}\cite{Ka1:2016}
In the setting of Theorem \ref{thm1}, if $i_k=d<r$, we have
\[
\Delta^L(k;{\rm \bf{i}})(\textbf{Y})=\sum_{(*)} a_{k_1,\cdots,k_d} B(l_1,k_1)B(l_2,k_2)\cdots B(l_d,k_d),
\]
\[ l_i:=
\begin{cases}
m-k_i+i & {\rm if}\ 1\leq k_i\leq r, \\
m-r-1+i & {\rm if}\ k_i\in\{\ovl{r},\cdots,\ovl{1}\}\cup\{0\},
\end{cases}
\]
\[ a_{k_1,\cdots,k_d}:=
\begin{cases}
2 & {\rm if}\ k_i=0\ {\rm for\ some}\ i\in[1,d], \\
1 & {\rm if}\ k_i\neq0\ {\rm for\ any}\ i\in[1,d] ,
\end{cases}
 \]
where $(*)$ is the condition for $k_i$ $(1\leq i\leq d)$ $:$ 
\begin{equation}\label{B-cond1}
1\leq k_1\leq k_2\leq \cdots\leq k_d\leq\ovl{1},\q k_i=k_{i+1}\ {\rm if\ and\ only\ if}\ k_i=k_{i+1}=0,
\end{equation}
\begin{equation}\label{B-cond2}
i\leq k_i\leq m-1+i\qq (1\leq i\leq r-m+1),
\end{equation}
\begin{equation}\label{B-cond3}
i\leq k_i\leq \ovl{d-i+1}\qq (r-m+2\leq i\leq d).
\end{equation}
\end{lem}

Set the subset $\mathbb{B}\subset\cY$ as
\[ \mathbb{B}:=\{B(l_1,k_1)B(l_2,k_2)\cdots B(l_d,k_d) | \{k_i\}\ {\rm satisfy}\ (\ref{B-cond1}),\ (\ref{B-cond2})\ {\rm and}\ (\ref{B-cond3}) \}. \]

Similar to Lemma \ref{Clem1}, we can verify the following lemma.

\begin{lem}\label{Blem1}

\begin{itemize} 
\item[(i)]For $k\in[1,r-1]$, we have $\tilde{e}_k B(l,k)=A_{l-1,k}\cdot B(l,k)=B(l-1,k+1)$ and $\tilde{e}_k B(l,\ovl{k+1})=A_{l,k}\cdot B(l,\ovl{k+1})=B(l,\ovl{k})$ in $\cY$. Furthermore, $\tilde{e}_r B(l,r)=A_{l-1,r}\cdot B(l,r)=B(l-1,0)$ and $\tilde{e}_r B(l,0)=A_{l,r}\cdot B(l,0)=B(l,\ovl{r})$.
\item[(ii)]For $X:=B(l_1,k_1)B(l_2,k_2)\cdots B(l_d,k_d)\in\mathbb{B}$ and $k\in[1,r-1]$, if $\tilde{e}_k X\neq0$, then there exists $j$ $(1\leq j\leq d)$ such that either the following $(a)$ or $(b)$ holds:

$(a)$ $k_j=k$, $k_{j+1}>k+1$ and
\[
 \tilde{e}_k X=B(l_1,k_1)\cdots B(l_{j-1},k_{j-1})B(l_j-1,k+1)B(l_{j+1},k_{j+1})\cdots B(l_d,k_d).  
\]
$(b)$ $k_j=\ovl{k+1}$, $k_{j+1}>\ovl{k}$ and
\[
\tilde{e}_k X=B(l_1,k_1)\cdots B(l_{j-1},k_{j-1})B(l_j,\ovl{k})B(l_{j+1},k_{j+1})\cdots B(l_d,k_d).  
\]
In particular, we have $\tilde{e}_k X\in \mathbb{B}$. 
\item[(iii)]If $\tilde{e}_r X\neq0$, then there exists $j$ $(1\leq j\leq d)$ such that either the following $(a)$ or $(b)$ holds:

$(a)$ $k_j=r$, $k_{j+1}\neq\ovl{r}$ and
\[
\tilde{e}_r X=B(l_1,k_1)\cdots B(l_{j-1},k_{j-1})B(l_j-1,0)B(l_{j+1},k_{j+1})\cdots B(l_d,k_d).
\]
$(b)$ $k_j=0$, $k_{j+1}>\ovl{r}$ and
\[
\tilde{e}_r X=B(l_1,k_1)\cdots B(l_{j-1},k_{j-1})B(l_j,\ovl{r})B(l_{j+1},k_{j+1})\cdots B(l_d,k_d).
\]
\end{itemize}

\end{lem}

\begin{lem}\label{Blem2}
For $X=B(l_1,k_1)B(l_2,k_2)\cdots B(l_d,k_d)\in\mathbb{B}$ and $s\in\mathbb{Z}$ $(0\leq s\leq r-m+1)$, we suppose that $1\leq k_1<\cdots<k_{d-s}\leq r$ and $k_{d-s+1},\ \cdots,\ k_{d-1},\ k_d\in\{0,\ovl{r},\cdots,\ovl{1}\}$, that is, $s:=\#\{i\in[1,d]|k_i\in\{0,\ovl{r},\cdots,\ovl{1}\}\}$. Then there exist non-negative integers $\{a(i,j)\}_{1\leq i\leq r-d+s,\ 1\leq j\leq r}$ such that
\begin{equation}\label{Blem2-1}
X=\tilde{e}^{a(r-d+s,1)}_1\cdots\tilde{e}^{a(r-d+s,r)}_r\cdots \tilde{e}^{a(2,1)}_1\cdots\tilde{e}^{a(2,r)}_r
 \tilde{e}^{a(1,1)}_1\cdots\tilde{e}^{a(1,d)}_d \frac{1}{Y_{m,d}}. 
\end{equation}
\end{lem}

\begin{proof}

Let us prove Lemma \ref{Blem2} by the induction on $s$. In the case $s=0$, we can prove (\ref{Blem2-1}) by the same argument in the proof of Lemma \ref{Clem2}. So, we assume $s>0$. Similar to (\ref{Clem2-2}), we see that
\begin{multline}\label{Blem2-2} B(l_1,k_1)B(l_2,k_2)\cdots B(l_{d-s},k_{d-s})=\\
B(m-1,\ovl{k'_1})B(m-2,\ovl{k'_2})\cdots B(m-r+d-s,\ovl{k'_{r-d+s}})\cdot\frac{1}{Y^2_{m-r+d-s,r}}, 
\end{multline}
where $\{k'_1,\ k'_2,\ \cdots,\ k'_{r-d+s} \}:=\{1,2,\cdots,r\}\setminus\{k_1,\ k_2,\ \cdots,\ k_{d-s} \}$ and $k'_1<k'_2<\cdots<k'_{r-d+s}$. Thus, the Laurent monomial $X$ can be written as follows:
\begin{multline}\label{Blem2-3}
X=B(m-1,\ovl{k'_1})B(m-2,\ovl{k'_2})\cdots B(m-r+d-s,\ovl{k'_{r-d+s}})\cdot\frac{1}{Y^2_{m-r+d-s,r}}\\
\cdot B(m+d-r-s,k_{d-s+1})\cdots B(m+d-r-2,k_{d-1})B(m+d-r-1,k_d).
\end{multline}

If $k_j\neq 0$ for any $j\in[1,d]$, then the monomial (\ref{Blem2-3}) is similar to (\ref{Clem2-3}). 
Therefore, in this case, we can show (\ref{Blem2-1}) in the same way as Lemma \ref{Clem2}. Hence, we may assume that $k_{d-s}=k_{d-s+1}=\cdots=k_{\zeta}=0$ and $\ovl{r}\leq k_{\zeta+1}<k_{\zeta+1}<\cdots<k_d\leq\ovl{1}$ for some integer $\zeta$ ($1\leq\zeta\leq s$). Then we can write as
\begin{multline*}
X=B(m-1,\ovl{k'_1})B(m-2,\ovl{k'_2})\cdots B(m+d-r-s,\ovl{k'_{r-d+s}})\cdot\frac{1}{Y^2_{m-r+d-s,r}} \\
\times B(m+d-r-s,0)B(m+d-r-s+1,0)\cdots B(m+d-r-s+\zeta-1,0) \\
\times B(m+d-r-s+\zeta,k_{\zeta+1})B(m+d-r-s+\zeta+1,k_{\zeta+2})\cdots B(m+d-r-1,k_d).
\end{multline*}
We put $\kappa=k'_{r-d+s}$ and suppose that there exists some integer $j$ $(\zeta+2\leq j\leq d+1)$ such that $|k_{j}|<\kappa\leq|k_{j-1}|$, where $|k_{d+1}|:=0$. By the action of
\begin{multline*}
F_0=(\tilde{f}^2_{r}\tilde{f}^2_{r-1}\cdots\tilde{f}^2_{|k_{\zeta+1}|}\tilde{f}_{|k_{\zeta+1}|-1})(\tilde{f}^2_{|k_{\zeta+1}|-2}\cdots \tilde{f}^2_{|k_{j-3}|}\tilde{f}_{|k_{j-3}|-1}\tilde{f}^2_{|k_{j-3}|-2}\\
\cdots\tilde{f}^2_{|k_{j-2}|}
\tilde{f}_{|k_{j-2}|-1})(\tilde{f}^2_{|k_{j-2}|-2}\cdots\tilde{f}^2_{|k_{j-1}|+1}\tilde{f}^2_{|k_{j-1}|}\tilde{f}_{|k_{j-1}|-1}
\cdots\tilde{f}_{\kappa+1}\tilde{f}_{\kappa}), 
\end{multline*}
each factor $B(m-d-r-s+i-1,\ovl{k_{i}})$ of $X$ becomes $B(m-d-r-s+i-1,\ovl{k_{i-1}-1})$ $(\zeta+2\leq i\leq j-1)$. Similarly, $B(m-d-r-s+\zeta,k_{\zeta+1})$ is changed to $B(m-d-r-s+\zeta,0)$. And we also find that $B(m+d-r-s,\ovl{k'_{r-d+s}})=B(m+d-r-s,\ovl{\kappa})$ and $B(m+d-r-s,0)$ are sent to $B(m+d-r-s,\ovl{r}) $ and $B(m+d-r-s+1,r)$ respectively. Hence, 
it follows from (\ref{Blem2-2}) and the same argument as in the proof of Lemma \ref{Clem2}, the monomial $F_0\cdot X$ has $s-1$ factors in the form $B(*,b)$ $(b\in\{0,\ovl{r},\cdots,\ovl{1}\})$. Thus, using the assumption of induction on $s-1$, we get
\[ F_0\cdot X=\tilde{e}^{a(r-d+s-1,1)}_1\cdots\tilde{e}^{a(r-d+s-1,r)}_r\cdots\tilde{e}^{a(1,1)}_1\cdots\tilde{e}^{a(1,d)}_d\frac{1}{Y_{m,d}}, \]
for some non-negative integers $\{a(i,j)\}$. And we can write
\[ X=F_0^{*}\tilde{e}^{a(r-d+s-1,1)}_1\cdots\tilde{e}^{a(r-d+s-1,r)}_r\cdots\tilde{e}^{a(1,1)}_1\cdots\tilde{e}^{a(1,d)}_d\frac{1}{Y_{m,d}}, \]
where $F_0^{*}$ is as in (\ref{Finv}). Therefore, $X$ is in the desired form in (\ref{Blem2-1}). 

\end{proof}

[{\sl Proof of Theorem \ref{thm1} for type ${\rm B}_r$ in the case $i_k=d<r$}]

All we have to show is $\mu(B^-(\Lm_d)_{u_{\leq k}})=\mathbb{B}$. Using Lemma \ref{Blem1}, we can prove $\mu(B^-(\Lm_d)_{u_{\leq k}})\subset\mathbb{B}$ in the same way as the case of type ${\rm C}_r$. The inverse inclusion $\mu(B^-(\Lm_d)_{u_{\leq k}})\supset\mathbb{B}$ is followed from Theorem \ref{kashidem} and Lemma \ref{Blem2} . \qed

Next, let us prove Theorem \ref{thm1} for type ${\rm B}_r$ in the case $i_k=d=r$.
We use the notation (\ref{bbbar}) and 
\[ B(l,r+1):=\frac{1}{Y_{l,r}}. \]

\begin{lem}\cite{Ka1:2016}
In the setting of Theorem \ref{thm1}, suppose that $i_k=r$. Then
\[
\Delta^L(k;{\rm \bf{i}})=\sum_{(*)} B(m-1,\ovl{k_1})B(m-2,\ovl{k_2})\cdots B(m-s,\ovl{k_s})B(m-s,r+1),
\]
where $(*)$ is the condition for $s\in\mathbb{Z}_{\geq0}$ and $k_i$ $(1\leq i\leq s)$ $:$ 
\begin{equation}\label{B-spin-set}
0\leq s\leq m-1,\qq 1\leq k_1<k_2<\cdots<k_s\leq r.
\end{equation}
\end{lem}

Note that in (\ref{B-spin-set}), if $s=0$, the summand of $\Delta^L(k;{\rm \bf{i}})$ is $B(m,r+1)$.
Set the subset $\mathbb{B}_{sp}\subset\cY$ as
\[ \mathbb{B}_{sp}:=\{B(m-1,\ovl{k_1})\cdots B(m-s,\ovl{k_s})B(m-s,r+1) | \{k_i\}\ {\rm and}\ s\ {\rm satisfy}\ (\ref{B-spin-set}) \}. \]

\begin{lem}\label{Bspin-lem1}

\begin{itemize} 
\item[(i)]We have $\tilde{e}_r B(l,r+1)=B(l-1,\ovl{r})B(l-1,r+1)$.
\item[(ii)]For $X:=B(m-1,\ovl{k_1})\cdots B(m-s,\ovl{k_s})B(m-s,r+1)\in\mathbb{B}_{sp}$ and $k\in[1,r-1]$, if $\tilde{e}_k X\neq0$ then there exists $j$ $(1\leq j\leq s)$ such that $\ovl{k_j}=\ovl{k+1}$, $\ovl{k}<\ovl{k_{j+1}}$ and
\[
\tilde{e}_k X=B(m-1,\ovl{k_1})\cdots B(m-j,\ovl{k})\cdots B(m-s,\ovl{k_s})B(m-s,r+1).  
\]
In particular, we have $\tilde{e}_k X\in \mathbb{B}_{sp}$. If $\tilde{e}_r X\neq0$ then $k_s\neq r$ and
\begin{equation}\label{ekXbsp3}
\tilde{e}_r X=B(m-1,\ovl{k_1})\cdots B(m-s,\ovl{k_s})B(m-s-1,\ovl{r})B(m-s-1,r+1).
\end{equation}
\end{itemize}

\end{lem}

\begin{proof}
(i) We obtain $\tilde{e}_r B(l,r+1)=A_{l-1,r}\frac{1}{Y_{l,r}}=B(l-1,\ovl{r})B(l-1,r+1)$.

(ii) For $k\in[1,r-1]$, we suppose that $\tilde{e}_k(X)\neq0$. If the Laurent monomial $X$ does not have factors in the form $\frac{1}{Y_{*,k}}$, then $\varepsilon_k (X)=0$, which implies $\tilde{e}_k(X)=0$. So, $X$ includes a factor in the form $\frac{1}{Y_{*,k}}$. The explicit form (\ref{bbbar}) means that $X$ has the factor $B(m-j, \ovl{k+1})$ for some $j\in[1,s]$ and does not have the factors in the form $B(*,\ovl{k})$. Then, we have $n_{e_k}=m-j$, and 
\[
\tilde{e}_k X=B(m-1,\ovl{k_1})\cdots B(m-j,\ovl{k})\cdots B(m-s,\ovl{k_s})B(m-s,r+1), 
\]
by Lemma \ref{Blem1} (i). Similarly, we see that if $\tilde{e}_r(X)\neq0$, then $k_s\neq r$ and $\tilde{e}_r X$ is described as in (\ref{ekXbsp3}). 
\end{proof}

[{\sl Proof of Theorem \ref{thm1} for type ${\rm B}_r$ in the case $i_k=r$}]

All we have to show is $\mu(B^-(\Lm_r)_{u_{\leq k}})=\mathbb{B}_{sp}$. Similar to the cases of type ${\rm C}_r$ and type ${\rm B}_r$ ($i_k<r$), we can prove $\mu(B^-(\Lm_r)_{u_{\leq k}})\subset\mathbb{B}_{sp}$ by using Lemma \ref{Bspin-lem1} (ii). 

Let us prove $\mu(B^-(\Lm_r)_{u_{\leq k}})\supset\mathbb{B}_{sp}$. Take an arbitrary monomial $X=B(m-1,\ovl{k_1})\cdots B(m-s,\ovl{k_s})B(m-s,r+1)\in\mathbb{B}_{sp}$. It follows from Lemma \ref{Bspin-lem1} that
\[ \tilde{e}_{k_1+1}\cdots\tilde{e}_{r-1}\tilde{e}_r B(m,r+1)=B(m-1,\ovl{k_1})B(m-1,r+1). \]
Using Lemma \ref{Bspin-lem1} repeatedly, we obtain
\begin{multline*}
(\tilde{e}_{k_s+1}\cdots\tilde{e}_{r-1}\tilde{e}_r) 
\cdots (\tilde{e}_{k_2+1}\cdots\tilde{e}_{r-1}\tilde{e}_r)(\tilde{e}_{k_1+1}\cdots\tilde{e}_{r-1}\tilde{e}_r) B(m,r+1)\\
=B(m-1,\ovl{k_1})\cdots B(m-s,\ovl{k_s})B(m-s,r+1)=X. 
\end{multline*}
Recall that $B(m,r+1)=\frac{1}{Y_{m,r}}=\mu(v_{\Lm_r})$, where $v_{\Lm_r}$ is the lowest weight vector of the crystal base $B(\Lm_r)$. Therefore, by Theorem \ref{kashidem}, we get $X\in\mu(B^-(\Lm_r)_{u_{\leq k}})$. Hence, we obtain $\mu(B^-(\Lm_r)_{u_{\leq k}})\supset\mathbb{B}_{sp}$. \qed

\subsection{Type ${\rm D}_r$}

Finally, we treat the case type ${\rm D}_r$. We shall use the notation (\ref{ddbar}). First, let us prove the theorem in the case $i_k=d<r-1$. Set the subset $\mathbb{B}\subset\cY$ as
\[ \mathbb{B}:=\{D(l_1,k_1)D(l_2,k_2)\cdots D(l_d,k_d) | \{k_i\}\ {\rm satisfy}\ (\ref{D-cond1}),\ (\ref{D-cond2})\ {\rm and}\ (\ref{D-cond3}) \}. \]
Similar to Lemma \ref{Clem1}, we can verify the following lemma.

\begin{lem}\label{Dlem1}

\begin{itemize} 
\item[(i)]For $k\in[1,r-1]$, we have $\tilde{e}_k D(l,k)=D(l-1,k+1)$ and $\tilde{e}_k D(l,\ovl{k+1})=D(l,\ovl{k})$ in $\cY$. Furthermore, $\tilde{e}_r D(l,r-1)=D(l-1,\ovl{r})$ and $\tilde{e}_{r} D(l,r)=D(l,\ovl{r-1})$.
\item[(ii)]For $X:=D(l_1,k_1)D(l_2,k_2)\cdots D(l_d,k_d)\in\mathbb{B}$ and $k\in[1,r-1]$, if $\tilde{e}_k X\neq0$, then there exists $j$ $(1\leq j\leq d)$ such that either the following $(a)$ or $(b)$ holds:

$(a)$ $k_j=k$, $k_{j+1}>k+1$ and
\[
 \tilde{e}_k X=D(l_1,k_1)\cdots D(l_{j-1},k_{j-1})D(l_j-1,k+1)D(l_{j+1},k_{j+1})\cdots D(l_d,k_d).
\]
$(b)$ $k_j=\ovl{k+1}$, $k_{j+1}>\ovl{k}$ and
\[
\tilde{e}_k X=D(l_1,k_1)\cdots D(l_{j-1},k_{j-1})D(l_j,\ovl{k})D(l_{j+1},k_{j+1})\cdots D(l_d,k_d).  
\]
In particular, we have $\tilde{e}_k X\in \mathbb{B}$. 
\item[$(iii)$]
If $\tilde{e}_r X\neq0$, then there exists $j$ $(1\leq j\leq d)$ such that either following $(a)$ or $(b)$ holds:

$(a)$ $k_j=r-1$, $k_{j+1}>\ovl{r-1}$ and
\[
\tilde{e}_r X=D(l_1,k_1)\cdots D(l_{j-1},k_{j-1})D(l_j-1,\ovl{r})D(l_{j+1},k_{j+1})\cdots D(l_d,k_d).
\]
$(b)$ $k_j=r$, $k_{j+1}>\ovl{r-1}$ and
\[
\tilde{e}_r X=D(l_1,k_1)\cdots D(l_{j-1},k_{j-1})D(l_j,\ovl{r-1})D(l_{j+1},k_{j+1})\cdots D(l_d,k_d).
\]
\end{itemize}

\end{lem}

\begin{lem}\label{Dlem2}
For $X:=D(l_1,k_1)D(l_2,k_2)\cdots D(l_d,k_d)\in\mathbb{B}$ and $s\in\mathbb{Z}$ $(0\leq s\leq r-m+1)$, we suppose that $1\leq k_1<\cdots<k_{d-s}\leq r-1$ and $k_{d-s+1},\ \cdots,\ k_{d-1},\ k_d\in\{\ovl{r},\cdots,\ovl{1}\}\cup\{r\}$, that is, $s:=\#\{i\in [1,d]|k_i\in\{\ovl{r},\cdots,\ovl{1}\}\cup\{r\}\}$. Then there exist non-negative integers $\{a(i,j)\}_{1\leq i\leq r-d+s,\ 1\leq j\leq r}$ such that
\begin{equation}\label{Dlem2-1}
X=\tilde{e}^{a(r-d+s,1)}_1\cdots\tilde{e}^{a(r-d+s,r)}_r\cdots \tilde{e}^{a(2,1)}_1\cdots\tilde{e}^{a(2,r)}_r
 \tilde{e}^{a(1,1)}_1\cdots\tilde{e}^{a(1,d)}_d \frac{1}{Y_{m,d}}. 
\end{equation}
\end{lem}

\begin{proof}

We set $\gamma:=m-r+d-s+1$. Similar to (\ref{Clem2-2}) and (\ref{Blem2-2}), we see that
\[D(l_1,k_1)D(l_2,k_2)\cdots D(l_{d-s},k_{d-s})=\frac{D(m-1,\ovl{k'_1})D(m-2,\ovl{k'_2})\cdots D(\gamma,\ovl{k'_{m-\gamma}})}{Y_{\gamma,r-1}Y_{\gamma,r}} \]
where $\{k'_1,\ k'_2,\ \cdots,\ k'_{m-\gamma} \}:=\{1,2,\cdots,r-1\}\setminus\{k_1,\ k_2,\ \cdots,\ k_{d-s} \}$ and $k'_1<k'_2<\cdots<k'_{m-\gamma}$. Thus,
\[
X=\frac{D(m-1,\ovl{k'_1})\cdots D(\gamma,\ovl{k'_{m-\gamma}}) D(l_{d-s+1},k_{d-s+1}) \cdots D(l_{d},k_{d})}{ Y_{\gamma,r-1}Y_{\gamma,r}}.\]
Putting $\kappa:=k'_{m-\gamma}$, we suppose that $k_{j}<\ovl{\kappa}\leq k_{j+1}$ for some $j$ $(d-s\leq j\leq d)$, where $k_{d+1}:=\ovl{1}$.

The remaining part of the proof is the same as in Lemma \ref{Blem2}, that is, acting the Kashiwara operators $\{\tilde{f}_i\}$ on $X$ properly, we can send the factor $D(l_{d-s+1},k_{d-s+1})$ to $D(l_{d-s+1}+1,r-1)$. Similar to the proof of Lemma \ref{Blem2}, using induction on $s$ and by its assumption, we see that the monomial $X$ can be written as in (\ref{Dlem2-1}). 

\end{proof}

[{\sl Proof of Theorem \ref{thm1} for type ${\rm D}_r$}]

The case $i_k=d<r-1$:

Similar to the cases of type ${\rm C}_r$ and type ${\rm B}_r$, we can verify $\mu(B^-(\Lm_d)_{u_{\leq k}})=\mathbb{B}$ by using Lemma \ref{Dlem1} and \ref{Dlem2}. 

The case $i_k=d=r-1$ or $r$:

Next, let us prove Theorem \ref{thm1} for type ${\rm D}_r$ of the cases $i_k=r-1$ and $i_k=r$. Since we can prove both cases in almost the same way, we shall prove the case $i_k=r$ only. We use the notation in (\ref{ddbar}) and (\ref{ddbar2}). Set the subset $\mathbb{B}^{(+)}_{sp}\subset\cY$ as
\[ \mathbb{B}^{(+)}_{sp}:=\{D(m-1,\ovl{k_1})\cdots D(m-s,\ovl{k_s})D(m-s,r+1) | \{k_i\}\ {\rm and}\ s\ {\rm satisfy}\ (\ref{D-spin-set}) \}. \]

Similar to Lemma \ref{Bspin-lem1}, we can verify the following lemma:

\begin{lem}\label{Dspin-lem1}

\begin{itemize} 
\item[(i)]We have $\tilde{e}_r D(l,r+1)=D(l-1,\ovl{r-1})D(l-2,\ovl{r})D(l-2,r+1)$.
\item[(ii)]For $X:=D(m-1,\ovl{k_1})\cdots D(m-s,\ovl{k_s})D(m-s,r+1)\in\mathbb{B}^{(+)}_{sp}$ and $k\in[1,r-1]$, if $\tilde{e}_k X\neq0$, then there exists $j$ $(1\leq j\leq s)$ such that $k_j=\ovl{k+1}$, $k_{j+1}<k$ and
\[
\tilde{e}_k X=D(m-1,\ovl{k_1})\cdots D(m-j,\ovl{k})\cdots D(m-s,\ovl{k_s})D(m-s,r+1).  
\]
In particular, we have $\tilde{e}_k X\in \mathbb{B}^{(+)}_{sp}$. If $\tilde{e}_r X\neq0$ then $k_s\leq r-2$ and
\[
\tilde{e}_r X=D(m-1,\ovl{k_1})\cdots D(m-s,\ovl{k_s})\\
\cdot D(m-s-1,\ovl{r-1})D(m-s-2,\ovl{r})D(m-s-2,r+1).
\]
\end{itemize}

\end{lem}

Now, let us complete the proof of Theorem \ref{thm1}. All we have to show is $\mu(B^-(\Lm_r)_{u_{\leq k}})=\mathbb{B}^{(+)}_{sp}$. Similar to the cases of type ${\rm B}_r$, ${\rm C}_r$ and ${\rm D}_r$ $(i_k<r-1)$, we can prove $\mu(B^-(\Lm_r)_{u_{\leq k}})\subset\mathbb{B}^{(+)}_{sp}$ by using Lemma \ref{Dspin-lem1} (ii). 

Let us prove $\mu(B^-(\Lm_r)_{u_{\leq k}})\supset\mathbb{B}^{(+)}_{sp}$. For an arbitrary monomial $X=D(m-1,\ovl{k_1})\cdots D(m-s,\ovl{k_s})D(m-s,r+1)\in\mathbb{B}^{(+)}_{sp}$, it follows from Lemma \ref{Dspin-lem1} that
\[ \tilde{e}_{k_1+1}\cdots\tilde{e}_{r-3}\tilde{e}_{r-2}\tilde{e}_r D(m,r+1)=D(m-1,\ovl{k_1})D(m-2,\ovl{r})D(m-2,r+1), \]
and
\[
 (\tilde{e}_{k_2+1}\cdots\tilde{e}_{r-3}\tilde{e}_{r-2}\tilde{e}_{r-1}) (\tilde{e}_{k_1+1}\cdots\tilde{e}_{r-3}\tilde{e}_{r-2}\tilde{e}_r) D(m,r+1)= \\
D(m-1,\ovl{k_1})D(m-2,\ovl{k_2})D(m-2,r+1).
\]

Using Lemma \ref{Dspin-lem1} repeatedly, we obtain
\begin{multline*}
(\tilde{e}_{k_s+1}\cdots\tilde{e}_{r-2}\tilde{e}_{r-1})(\tilde{e}_{k_{s-1}+1}\cdots\tilde{e}_{r-2}\tilde{e}_r) 
\cdots (\tilde{e}_{k_2+1}\cdots\tilde{e}_{r-2}\tilde{e}_r)(\tilde{e}_{k_1+1}\cdots\tilde{e}_{r-2}\tilde{e}_r) D(m,r+1)\\
=D(m-1,\ovl{k_1})D(m-2,\ovl{k_2})\cdots D(m-s,\ovl{k_s})D(m-s,r+1)=X. 
\end{multline*}
Recall that $D(m,r+1)=Y_{m,r}=\mu(v_{\Lm_r})$, where $v_{\Lm_r}$ is the lowest weight vector of the crystal base $B(\Lm_r)$. Therefore, by Theorem \ref{kashidem}, we get $X\in\mu(B^-(\Lm_r)_{u_{\leq k}})$. Hence, we obtain $\mu(B^-(\Lm_r)_{u_{\leq k}})\supset\mathbb{B}^{(+)}_{sp}$. \qed

Due to the above arguments, we have completely shown Theorem \ref{thm1}.

\subsection{A counter example}

In this final subsection, we shall see an example such that $\Delta(k;\textbf{i})(\textbf{Y})$ can not be described as the total sum of monomials in any Demazure crystals.

\begin{ex}

Let $G$ be the simple algebraic group of type ${\rm C}_3$, and we set $u:=(s_1s_2s_3)^3\in W$, ${\rm \bf{i}}:=(1,2,3,1,2,3,1,2,3)$ and $k=3$. In the notation of Theorem \ref{thm1}, we have $m=3$, $d=i_k=3$, and $i_k$ belongs to $(m-2)$th cycle. Hence, this setting does not satisfy the condition in Theorem \ref{thm1} such that $i_k$ belongs to $(m-1)$th cycle. In this case, following \cite{KaN2:2016}, we obtain
\begin{eqnarray*}
\Delta^L(3;{\rm \bf{i}})(\textbf{Y})&=&\left(\frac{1}{Y_{2,3}Y_{3,3}}+\frac{Y_{1,3}}{Y^2_{2,2}Y_{3,3}}+\frac{Y_{1,3}Y_{2,3}}{Y^2_{2,2}Y^2_{3,2}}+\frac{2Y_{1,2}}{Y_{2,1}Y_{2,2}Y_{3,3}}+\frac{Y^2_{1,2}}{Y_{1,3}Y^2_{2,1}Y_{3,3}}
+\cdots+\frac{Y^2_{1,1}}{Y_{1,3}Y^2_{2,1}Y_{2,3}}\right)\\
&+&2\left(\frac{Y_{1,2}}{Y_{2,1}Y_{3,1}Y_{3,2}}+\frac{Y_{1,1}}{Y_{3,1}Y_{3,2}}+\frac{Y_{1,1}Y_{2,1}}{Y_{2,2}Y_{3,2}}+\frac{Y_{1,1}Y_{1,2}}{Y_{1,3}Y_{3,2}}+\frac{Y_{1,1}Y_{1,2}Y_{2,2}}{Y_{1,3}Y_{2,3}Y_{3,1}}\right).
\end{eqnarray*}

The set of the monomials in the part $(\frac{1}{Y_{2,3}Y_{3,3}}+\frac{Y_{1,3}}{Y^2_{2,2}Y_{3,3}}+\frac{Y_{1,3}Y_{2,3}}{Y^2_{2,2}Y^2_{3,2}}+\frac{2Y_{1,2}}{Y_{2,1}Y_{2,2}Y_{3,3}}+\frac{Y^2_{1,2}}{Y_{1,3}Y^2_{2,1}Y_{3,3}}
+\cdots+\frac{Y^2_{1,1}}{Y_{1,3}Y^2_{2,1}Y_{2,3}})$ coincides with a monomial realization of the Demazure crystal $B(2\Lambda_3)_{s_1s_2s_3}$. The set of the monomials in $(\frac{Y_{1,2}}{Y_{2,1}Y_{3,1}Y_{3,2}}+\frac{Y_{1,1}}{Y_{3,1}Y_{3,2}}+\frac{Y_{1,1}Y_{2,1}}{Y_{2,2}Y_{3,2}}+\frac{Y_{1,1}Y_{1,2}}{Y_{1,3}Y_{3,2}}+\frac{Y_{1,1}Y_{1,2}Y_{2,2}}{Y_{1,3}Y_{2,3}Y_{3,1}})$ coincides with a monomial realization of certain subset of the crystal $B(2\Lambda_2)$, but it does not coincide with any Demazure crystals $B(2\Lambda_2)_w$ $(w\in W)$.
\end{ex}





%

%

\end{document}